\documentclass{article}
\usepackage[utf8]{inputenc}
\usepackage[english]{babel}
\usepackage{blindtext}
\usepackage{amssymb}
\usepackage{amsmath}
\usepackage{comment}
\usepackage{mleftright}
\usepackage{xcolor}
\usepackage{amsthm}
\usepackage{amsmath}
\numberwithin{equation}{section}

\newtheorem{theorem}{Theorem}[section]

\newtheorem{lemma}[theorem]{Lemma}
\newtheorem{proposition}[theorem]{Proposition}

\theoremstyle{definition}

\newcommand{\bean}{\begin{eqnarray*}}
\newcommand{\eean}{\end{eqnarray*}}
\newcommand{\be}{\begin{equation}}

\newcommand{\ee}{\end{equation}}
\newcommand{\bd}{\begin{displaymath}}
\newcommand{\ed}{\end{displaymath}}
\newcommand{\ep}{\varepsilon}

\newcommand{\cals}{{\mathcal{S}}}

\newcommand{\beq}{\begin{equation}}
\newcommand{\eeq}{\end{equation}}
\newcommand{\bea}{\begin{eqnarray}}
\newcommand{\eea}{\end{eqnarray}}

\newcommand{\EUep}{E_\ep}

\newcommand{\R}{\mathbb{R}}
\newcommand{\e}{\varepsilon}

\newcommand{\C}{\mathbb{C}}

\newcommand{\calF}{\mathcal{F}}
\newcommand{\calV}{\mathcal{V}}

\def\logeps{{|\!\log\varepsilon|}}

\newcommand{\pp}{\partial}

\newcommand{\Z}{{\mathbb Z}}
\newcommand{\N}{{\mathbb N}}

\newcommand{\mass}{{\bf M}}

\newcommand{\calB}{{\mathcal{B}}}
\newcommand{\calD}{{\mathcal{D}}}
\newcommand{\calR}{{\mathcal{R}}}

\newcommand{\calL}{{\mathcal{L}}}
\newcommand{\calG}{{\mathcal{G}}}
\newcommand{\calH}{{\mathcal{H}}}

\newcommand{\calM}{{\mathcal{M}}}
\newcommand{\dist}{\operatorname{dist}}
\newcommand{\spt}{\operatorname{spt}}
\def\rest{\hskip 1pt{\hbox to 10.8pt{\hfill
\vrule height 7pt width 0.4pt depth 0pt\hbox{\vrule height 0.4pt
width 7.6pt depth 0pt}\hfill}}}

\newcommand{\vol}{\,\mbox{vol}}

\definecolor{darkgreen}{rgb}{0,0.55,0} 

\title{Solutions of the Ginzburg-Landau equations with vorticity concentrating near a nondegenerate geodesic}
\date{\today}

\usepackage[capitalise]{cleveref}

\usepackage{authblk}

\begin{document}
\renewcommand\Authfont{\small}
\renewcommand\Affilfont{\itshape\footnotesize}

\author[1]{Andrew Colinet\footnote{andrew.colinet@mail.utoronto.ca}}
\author[1]{Robert L. Jerrard\footnote{rjerrard@math.toronto.edu}}
\author[2]{Peter Sternberg\footnote{sternber@indiana.edu}}
\affil[1]{Department of Mathematics, University of Toronto, Toronto, ON M5S 2E4 Canada}
\affil[2]{Department of Mathematics, Indiana University, Rawles Hall
831 East 3rd St.,
Bloomington, IN 47405}

\maketitle

\noindent {\bf Abstract}
It is well-known that under suitable hypotheses, for a sequence
of solutions of the (simplified) Ginzburg-Landau equations
$-\Delta u_\ep +\ep^{-2}(|u_\ep|^2-1)u_\ep = 0$,
the energy and vorticity concentrate as $\ep\to 0$ around a 
codimension $2$ stationary varifold --- a (measure theoretic) minimal surface.
Much less is known about the question of whether, given a 
codimension $2$ minimal surface, there exists a sequence of
solutions  for  which the given minimal surface is
the limiting concentration set.  The corresponding question
is very well-understood for minimal hypersurfaces and the 
scalar Allen-Cahn equation, and for the Ginzburg-Landau equations
when the minimal surface is locally area-minimizing, but otherwise 
quite open.

We consider this question on a $3$-dimensional closed Riemannian manifold $(M,g)$,
and we prove that any embedded nondegenerate closed geodesic can be realized as the asymptotic energy/vorticity concentration set of a sequence of solutions of the 
Ginzburg-Landau equations.


\section{Introduction}

In this paper we construct certain geometrically meaningful solutions of the Ginzburg-Landau equations
\beq\label{GL}
 - \Delta u_\ep + \frac 1{\ep^2}(|u_\ep|^2-1)u_\ep = 0 
\eeq
for $u_\ep:M\to \C$, where $(M,g)$ is a closed $n$-dimensional Riemannian manifold, 
with $n=3$ in our main results.
Such solutions are critical points of the Ginzburg-Landau functional
\[
\EUep(u_\ep) := \frac 1{\pi\logeps}\int_M e_\ep(u_\ep) \vol_g, \qquad 
e_\ep(u_\ep) := \frac 12 |\nabla u_\ep|^2
+ \frac 1{4\ep^2}(|u_\ep|^2-1)^2.
\]
If  $M$ is simply connected, then
given a sequence of solutions $(u_\ep)$ of \eqref{GL} 
satisfying the energy bound 
\beq\label{uniformE}
\EUep(u_\ep)\le C,
\eeq 
the rescaled energy density
$\logeps^{-1}e_\ep(u_\ep)$ is known to concentrate as $\ep\to 0$,
after possibly passing to a subsequence, around an
$(n-2)$-dimensional stationary varifold --- a weak, measure-theoretic
minimal surface. This is proved in an appendix in \cite{Stern}, following earlier results 
in simply-connected Euclidean domains, such as those in
\cite{BBO, LinRiv1, LinRiv2}.
Similar but more complicated results hold
when $M$ is not simply connected; in this case,
the limiting energy measure may have a diffuse part, but
any concentrated part must again be an $(n-2)$-dimensional  stationary varifold.


In this paper we address a sort of converse question: 

\begin{center}
{\it When can a given  codimension $2$
minimal surface  be realized as the energy concentration set of a sequence of 
solutions of \eqref{GL}? }
\end{center}
A first answer is provided by Gamma-convergence results, see \cite{JSon,ABO},
that relate the Ginzburg-Landau functional and, roughly speaking, the $(n-2)$-dimensional area (with multiplicity) of a limiting vorticity concentration set, where the vorticity associated to a wave function $u$, 
denoted $Ju$, is the $2$-form defined by
\beq\label{Ju.def}
Ju := du^{1}\wedge{}du^{2}, \qquad
\mbox{where $u=u^{1}+iu^{2}$ and $u^1, u^2$ are real-valued.}
\eeq
(We will also sometimes refer to $Ju$ as the Jacobian of $u$.)
These results  imply as a general principle that one should be  able to find solutions  $u_\ep$ of \eqref{GL} whose energy and vorticity
concentrate around a  {\em locally area-minimizing} minimal surface of codimension 2. In the Euclidean setting,
specific
instances of this general principle, for particular compatible choices of boundary conditions on 
the minimal surface and the solutions $u_\ep$ of \eqref{GL}, 
have been established in \cite{ABO, Sand,MSZ}. However,  arguments based on Gamma-convergence are of limited use for capturing the behaviour of non-minimizing critical points. 

The corresponding question is also very well-understood for minimal {\em hypersurfaces} and the Allen-Cahn equation, {\em i.e.} the scalar counterpart of \eqref{GL}, see for example \cite{KS, PR,Kow, dPKW} among many others. Many of these results are based on gluing techniques and elliptic PDE arguments, which can be used to construct a great variety of solutions and establish detailed descriptions of them. These techniques seem to be hard to implement for the Ginzburg-Landau equation in $3$ or more dimensions.

A particularly basic case in which our question remains open
concerns the Ginzburg-Landau equation \eqref{GL} on a  smooth bounded domain $\Omega\subset \R^3$ 
containing an unstable geodesic with respect
to natural boundary conditions, {\em i.e.}  a line segment  in $\Omega$ meeting $\partial \Omega$ orthogonally at both ends, admitting perturbations that decrease the arclength quadratically, and satisfying a natural nondegeneracy condition.

In this situation one would like to prove the existence of a sequence $(u_\ep)$  of solutions of the Ginzburg-Landau equations, also with natural (Neumann) boundary conditions, whose energy and vorticity concentrate around the given line segment. Such solutions
would satisfy
\beq\label{energyL}
\lim_{\ep\to 0} \EUep(u_\ep) = L =: \mbox{  the length of the geodesic}.
\eeq
Partial progress toward this goal was achieved in \cite{JSt}, which develops 
a general framework for using Gamma-convergence to study
convergence, not of {\em critical points}, but of {\em critical values},
then uses this framework to prove the existence of solutions of \eqref{GL},
in the situation described above, that satisfy \eqref{energyL}, but without control
over the limiting concentration set. An example in the same paper (Remark 4.5)
shows that the general framework is too weak to  characterize asymptotic behaviour of critical points --- in this context, to determine where the energy and vorticity concentrate.
For this, more detailed information about the sequence of solutions is needed.

The results of \cite{JSt} were extended to the Riemannian setting
in the Ph.D. thesis of Jef\mbox{}frey Mesaric in \cite{Mesaric} which, starting with
a nondegenerate unstable closed geodesic on a closed, oriented $3$ dimensional Riemannian manifold $(M,g)$,
uses machinery from \cite{JSt} to construct solutions to \eqref{GL}
satisfying  \eqref{energyL}.
Again, this result  does not establish whether the energy of the solutions concentrates along the geodesic.

In the main result of this paper, we fill in this gap in the Riemannian case.
Our main result is the following theorem.

\begin{theorem}\label{thm:1a}
Let $(M,g)$ be a closed oriented $3$-dimensional Riemannian manifold, and 
let $\gamma$ be a closed, embedded, nondegenerate geodesic of length $L$.
Assume in addition that 
$\gamma = \partial S$ in the sense of 
Stokes' Theorem for some $2$-dimensional submanifold $S$ of $M$. 

Then there exists $\varepsilon_{1}>0$ such that for every $0<\varepsilon <\varepsilon_1$,
there is a solution $u_{\varepsilon}$ of the Ginzburg-Landau equation \eqref{GL}
such that
\begin{equation*}
\frac 1 \pi \int_M \varphi \wedge Ju_\varepsilon  \to  \int_\gamma \varphi 
\qquad\mbox{ for every smooth $1$-form $\varphi$ on $M$}
\end{equation*}
and
\[
\frac 1{\pi |\log\varepsilon|} \int_M \phi\, e_\varepsilon(u_\varepsilon)
\to \int_\gamma \phi \, d\calH^1 \qquad\mbox{ for every }\phi \in C^\infty(M)
\]
as $\varepsilon \to 0$, 
where $\calH^k$ denotes $k$-dimensional Hausdorff measure.
\end{theorem}

In fact we will prove a slightly stronger result; see Theorem \ref{thm:1} for the full statement.

We briefly sketch the main ideas, {\em not} in the
order in which they appear in the body of the paper. 
Terminology such as ``nondegenerate" and ``stationary varifold" are defined 
in Section \ref{sec:backnot} below.

\begin{itemize}
\item In Section 4 we show that for any $\delta>0$, 
there exists $\ep_0>0$ such that for $0<\ep<\ep_0$ and any $\tau>0$,
one can find a solution $u_\ep$ of the Ginzburg-Landau
heat flow whose vorticity is initially concentrated near the geodesic $\Gamma:=\gamma([0,L))$,
and such that 
\[
L-\delta \le \EUep(u_\ep(\cdot, t)) \le L+\delta \qquad\mbox{ for all }t\in [0,\tau].
\]
See Proposition \ref{prop:goodtrajectory}.
This relies heavily on tools developed in the earlier
papers \cite{JSt, Mesaric}.

The main point of the proof of Theorem \ref{thm:1a} is to strengthen this by showing 
that for such solutions, if $\ep$ and $\delta$ are 
small enough, the vorticity $\frac 1 \pi J u_\ep(\cdot,t)$ does not stray very far from
$\Gamma$ for any $t\in [0,\tau]$. 

\item We carry this out in Section \ref{sec:5}, using an argument by contradiction 
and passing to limits
to obtain a stationary $1$-dimensional varifold that is close, but not equal, to the varifold associated to $\Gamma$. This argument requires, among other ingredients,
an extension to the Riemannian setting of an important theorem of Bethuel, Orlandi, and Smets \cite{BOS}. The extension we need is stated in Theorem 
\ref{T:BOS} and is proved in a companion paper, see \cite{Colinet}.
The stationary varifold satisfies additional good properties, notably including lower density bounds.

\item 
To obtain a contradiction, we prove  that this stationary varifold cannot exist.
This is the content of Proposition \ref{lem:noVstar}, which
is a measure theoretic strengthening of the classical fact that a nondegenerate closed geodesic is isolated; it is the only closed geodesic in a tubular neighborhood of itself.  The proof relies, among other ingredients, on results from \cite{AA}  about the structure of stationary $1$-dimensional varifolds on Riemannian manifolds.

\end{itemize} 

We believe that something like Theorem \ref{thm:1a} should be valid in much greater generality, including on higher-dimensional manifolds and on smooth, bounded  subsets of $\R^n$, $n\ge 3$, with natural boundary conditions both for the geodesic $\Gamma$ (or codimension $2$ minimal surface, for $n\ge 4$) and the Ginzburg-Landau equation. 
Our proof does not adapt in a straightforward way to either of these settings.

\begin{itemize}
\item Our strategy requires a sufficiently good version of Theorem \ref{T:BOS}. On a bounded set $\Omega\subset \R^n$, even for $n=3$, such a result is not known. If $\Omega$ is convex, a result of this type for the scalar parabolic Allen-Cahn equation was proved several years ago in \cite{MizunoTonegawa}.
A similar strategy could probably be pursued for the Ginzburg-Landau heat flow, but global convexity is not a natural assumption for any analog of Theorem \ref{thm:1a}. 

\item Our reliance on results from \cite{AA} about stationary $1$-dimensional varifolds would seriously complicate any effort to adapt our argument to dimensions $n\ge 4$,
where one would confront stationary varifolds of dimension $n-2\ge 2$.
\end{itemize}

\section{Background and notation}\label{sec:backnot}

\subsection{Geometric notions regarding a non-degenerate geodesic}\label{sec:geom.notions}\hspace{5pt}{}\vspace{5pt}

Throughout this document we use $M$ or $(M,g)$ to denote a closed oriented three dimensional Riemannian manifold where ``closed'' means compact
and without boundary.
We let $TM$ be the bundle over $M$ whose fiber $T_{p}M$ at $p\in{}M$ is the tangent space to $M$ at $p$.
We use the notation $(\cdot,\cdot)_{g}$ to denote the inner product on $TM$ given by $g$.
We also use $\left|\cdot\right|_g$ to denote the corresponding norm, where we will omit mention of $g$ when no confusion will arise.
We write $\vol_g$ to denote the Riemannian volume form associated to the metric $g$.

We will write $r_0>0$ to be a number, fixed throughout this paper, such that
\beq\label{rzero.def}
r_0 < \frac 12\, (\mbox{injectivity radius of $M$}).
\eeq

Throughout this paper, a central role will be played by a geodesic $\gamma$ that we take to be parametrized by arclength. That is, we will assume the existence of an injective map $\gamma:\R/L\Z\to M$ whose range consists of a simple closed curve $\Gamma :=  \{ \gamma(t) : t\in \R/L\Z\} $ of length $L$
such that
\begin{equation}
 |\gamma'|=1, \qquad \nabla_{\gamma'}\gamma' = 0\qquad\mbox{ everywhere in }\R/L\Z.\label{geoODE}
\end{equation}

We will insist that this curve $\Gamma$ bounds an orientable smooth surface $S_\Gamma\subset M$, i.e. \begin{equation}
    \Gamma=\partial S_{\Gamma}.\label{gambdS}
\end{equation}

We introduce here the notation 
\[
d_\Gamma(x):=\dist(x,\Gamma):=\inf \left\{
\int_0^1 |\lambda'(t)|dt : \lambda\in Lip([0,1]; M), \lambda(0) = x, \lambda(1)\in \Gamma\right\}
\]
as well as
\begin{equation}
 K_{r}:=\left\{x\in{}M:d_\Gamma(x)<r\right\}   \label{neigh}
\end{equation}
for a neighborhood of $\Gamma.$

For $t\in \R/L\Z$, we then let
\[
N_{\gamma(t)}\Gamma :=  \{ u \in T_{\gamma(t)}M : (u, \gamma'(t))_g  = 0\}.
\]
A \emph{normal vector field along $\gamma$} is a map $\xi:\R/L\Z\to TM$
such that $\xi(t)\in N_{\gamma(t)}\Gamma$ for every $t$.
We also introduce the coordinates $\psi:B_{r}(0)\times(0,L)\to{}K_{r}$
defined by
\beq\label{TubularCoordinates}
\psi(y,t):=\exp_{\gamma(t)}\left(\sum_{i=1}^{2}y^{i}\Xi_{i}(t)\right)
\eeq
where $\Xi_{1},\Xi_{2}$ are fixed normal vector
fields which are orthogonal for each $t\in(0,L)$.
We note for $r<r_0$, this map is smoothly invertible. 
For future use, we will use the notation
$\psi^{-1}(x) = (y(x), \tau(x)) \in B_r(0)\times (0,L)$, so that for $x\in K_r$,
\beq\label{y.tau.def}
\psi(y,t)=x \qquad\iff \qquad y(x) = y \quad \mbox{ and }\quad \tau(x) = t.
\eeq
We observe that the mapping $\tau$ simply assigns to an $x\in K_r$ the parameter value $t$ corresponding to the closest point on $\Gamma$ to $x$.

Given two normal vector fields along $\gamma$, denoted by $\xi, \tilde{\xi}$, we can define the $L^2$ inner
product in the natural way:
\[
(\xi, \tilde{\xi})_{L^2}  := \int_{\R/L\Z} (\xi(t), \tilde{\xi}(t))_g\, dt.
\]
We will write $L^2(N\Gamma)$ to denote the space of square integrable
normal vector fields, a Hilbert space with the above inner product.\vspace{5pt}

For $\xi\in L^2(N\Gamma)$, we will use the notation
\begin{equation}
\gamma_\xi(t) := \exp_{\gamma(t)} \xi(t),
\label{gamexp}
\end{equation}
where $\text{exp}$ denotes the exponential map.

We next recall the {\em Jacobi operator} $L_J$ which acts on smooth normal vector fields $\xi$ along $\gamma$, 
and is defined by
\begin{equation}
    L_J \xi := - \xi'' + R(\xi, \gamma')\gamma',\label{jacobi}
\end{equation}
where $R$ denotes the curvature tensor. We say that a geodesic is {\em nondegenerate}  if $0$ is not an eigenvalue of $L_J$.

With this notion in hand, we add another crucial hypothesis on the geodesic by assuming henceforth that
\beq\label{gamma.def}
\mbox{$\gamma:\R/L\Z\to M$ is a simple, closed, nondegenerate geodesic with $|\gamma'|\equiv 1$, }
\eeq

One says that $\gamma$ has {\em finite index} if the total number (algebraic multiplicity)
of negative eigenvalues of $L_J$ is finite. Since $M$ is closed,
this is always true, as a consequence of standard Sturm-Liouville theory.
Our standing assumption \eqref{gamma.def} that $\gamma$ is nondegenerate
then imples
there exists some $\ell\ge 0$ and a nondecreasing sequence of eigenvalues 
\beq\label{index.ell}
\lambda_1 \le 
\ldots \le \lambda_\ell < 0 < \lambda_{\ell+1} \le\ldots
\eeq
of $L_J$, together with an associated orthonormal basis of $L^2(N \Gamma)$
consisting of (smooth) eigensections $\{ \xi_j\}_{j=1}^\infty$. We will always assume that
$\ell>0$, since otherwise the results presented here admit much simpler proofs.
We define
\beq\label{H+-.def}
H_- := \mbox{span}\{\xi_1,\ldots \xi_\ell\},
\qquad
H_+ := H_-^\perp.
\eeq
We will say that $\xi$ is Lipschitz, and we will write $\xi\in Lip$, if $\gamma_\xi$ is Lipschitz continuous. It is clear that 
\[
H_-(r_0) := \{ \xi\in H_-: \| \xi \|_{L^\infty}\le r_0\} \subset Lip
\]
for $r_0$ and $H_-$  from  \eqref{rzero.def} and \eqref{H+-.def} respectively.\vspace{5pt}

The standard fact that the Jacobi operator, cf. \eqref{jacobi}, is the second variation of arclength, together with the definition \eqref{H+-.def} of $H_-$,  implies that there exist
$c_0, r_0>0$ such that for all $\|\xi\|_{L^\infty}\le r_0$ one has
\beq\label{gxi.length}
\begin{aligned}
\int_{\R/L\Z} |\gamma_\xi'(t)|\, dt &
\le   L- c_0 \|\xi \|_{L_2}^2\quad\mbox{ if }\xi\in H_-(r_0)
\\
\int_{\R/L\Z} |\gamma_\xi'(t)|\, dt &
\ge  L+ c_0 \|\xi\|_{L_2}^2\quad\ \, \  \mbox{ if }\xi\in H_+\cap Lip .
\end{aligned}
\eeq

\subsection{Forms and currents}

We denote, for $k\in\N\cup\left\{0\right\}$, the space of smooth $k$-forms on $M$ by
\[
\calD^{k}(M):=\left\{\phi\in{}C^{\infty}(M;\wedge^{k}M)\right\}
\]
where $\wedge^{k}M$ is an abbreviated notation for $\wedge^{k}T^{*}M$.
We denote the dual space of $\calD^{k}(M)$, for $k\in\N\cup\left\{0\right\}$, by
\[
\calD_k(M) := \{\mbox{$k$-currents on $M$} \}.
\]
We refer to the elements of $\calD_{k}(M)$ as \emph{$k$-currents}.
For a $k$-current $T$, we define
\[
\mbox{ the {\em mass} of $T$} =\mass(T) := \sup\{ T(\phi) : \|\phi\|_\infty \le 1\} \in [0,+\infty].
\]
We will be most interested in $1$-currents. A basic class of examples consists of
$1$-currents we shall write as $T_{\lambda}$ whose action on $\phi\in \calD^1(M)$ takes the form
\beq\label{Tlambda.def}
T_\lambda(\phi) := \int_{\lambda}\phi ,\qquad\mbox{ where }\lambda:(a,b)\to M
\mbox{ is a Lipschitz curve}.
\eeq
We will say a $1$-current is {\em integer multiplicity rectifiable} if it can be written as a finite or
countable sum of currents of the form \eqref{Tlambda.def}. We will write
\[
\calR_1(M) := \{ T\in \calD_1(M) \ : \ \mass(T)<\infty, \ \  \mbox{$T$ is  integer multiplicity rectifiable }\}.
\]
For a $1$-current $J$, we write $\|J\|$ to denote the associated total variation measure, defined through its action on continuous, nonnegative functions $f:M\to\R$ via
\begin{equation}
 \int f \, d\|J\| := \sup \{ J(\phi): \phi\in \calD^1(M),\; |\phi|_g\le f\}.  \label{tovar}  
\end{equation}

For a $k$-current $S$, the {\em boundary of $S$} is the $(k-1)$-current $\partial S$ defined by
\[
\partial S(\phi) := S (d\phi), \qquad \mbox{for all}\;\phi \in \calD^{k-1}(M).
\]
We define
\[
\calF_1'(M) := \{  T\in \calD_1(M):  \ \ T=\partial S\mbox{ for some }S\in \calD_{2}(M),\, \mass(S)<\infty \}.
\]
and for $T\in \calF_1'(M)$, we will write
\[
\|  T\|_{\calF} := \inf\{ \mass(S) : T=\partial S\} .
\]
We also define
\[
\calR_1'(M) :=   \calR_1(M)\cap \calF_1'(M).
\]

We note that the $1$-current $T_\gamma$ associated with the geodesic $\gamma$ via \eqref{Tlambda.def}, in particular, bounds a finite mass $2$-current; that is,
\begin{equation}
 T_\gamma\in \calR_1'(M),\label{Tbounds}   
\end{equation}
in light of the assumption \eqref{gambdS}.

Lastly, we will at times wish to identify the Jacobian ({\em i.e.} vorticity) of a map $u\in H^1(M;\C)$ with an element of $\calD_1(M)$, which we
denote $\star J(u)$, and which is defined through its action on
$1$-forms $\phi$ by
\begin{equation}
 \star J(u)(\phi) = \int \phi \wedge J(u),\label{starJu}
\end{equation}
where $J(u)=du^{(1)}\wedge du^{(2)}$ for $u=u^{(1)}+iu^{(2)}$ where $u^{(1)},u^{(2)}$ are real-valued.

\subsection{Gamma-limit of the Ginzburg-Landau functional}

Below we state the version we will need of  standard Gamma-convegence results for the Ginzburg-Landau functional.

We first fix the notation $V = \calF_1'(M)$, with the flat norm $\| v\|_V := \|v\|_{\calF}$.
We also define the functional
\beq\label{EV.def}
E_V(T) := 
\begin{cases}
\mass(T) &\mbox{ if }T\in \calR_1'(M)\\
+\infty &\mbox{ if not}.
\end{cases}
\eeq
Thus $E_V$ is an extension to $V$ of the ``arclength functional" in the sense that if  $\lambda:(a,b)\to M$ is an injective Lipschitz continuous curve and $T_\lambda$ is the corresponding current, then $E_V(T_\lambda) = \mbox{ arclength of }Image(\lambda)$.

The following result is deduced in \cite{Mesaric}, Theorem 5.1 from corresponding Euclidean results, cf. \cite{ABO, JSon}.

\begin{theorem}\label{thm:Mesaric-Gamma} Let $(M,g)$ be a closed $3$-dimensional Riemannian manifold.

\smallskip

{\bf 1}. Let  $(u_\ep)_{0<\ep<\ep_0}$ be a sequence in $H^1(M;\C)$. If there exists $C>0$ such that 
$\EUep(u_\ep)\le C$ for all $\ep\in (0,\ep_0)$, then $(\frac 1\pi \star Ju_\ep)_{0<\ep<\ep_0}$ is precompact in V, and
any limit as $\ep\to 0$ belongs to $\calR_1'(M)$

\smallskip

{\bf 2}. Let  $(u_\ep)_{0<\ep<\ep_0}$ be a sequence in $H^1(M;\C)$. If $T\in V$ and\\ $\|\frac 1\pi \star Ju_\ep - T\|_\calF \to 0$
as $\ep\to 0$, then $\liminf_{\ep\to 0} \EUep(u_\ep)\ge E_V(T)$.

\smallskip

{\bf 3}. For any $T\in V$, there exists a sequence $(u_\ep)_ {0<\ep<\ep_0}$ in $H^1(M;\C)$
such that $\|\frac 1\pi \star Ju_\ep - T\|_\calF \to 0$ and $\limsup_{\ep\to 0} \EUep(u_\ep)\le E_V(T)$.

\end{theorem}

The geodesic $\gamma$ is a saddle point of the arclength with respect to smooth perturbations, as reflected in \eqref{gxi.length}. For use in combination with Theorem \eqref{thm:Mesaric-Gamma}, 
one needs to identify a sense in which the corresponding current $T_\gamma$ is a saddle point of $E_V$. We defer a discussion of this and related issues to Section \ref{sec:4}.

\subsection{Varifolds}\label{sec:Var}
We briefly recall the definition of a rectifiable $1$-varifold and introduce some notation that will be used later.
After doing this we will introduce the definition of a general $1$-varifold.
We note that the general definition will only be used in the proof of Proposition $\ref{lem:noVstar}$.
For general varifolds we will follow \cite{AA} with some terminology from \cite{LS}.

For any $1$-dimensional rectifiable set 
$\Sigma$,  basic theory (see for example \cite{LS}, Lemma 11.1) shows that there exists a
countable family of $C^1$ curves $(\Lambda_j)_{j\in \N}$ in $M$ such that
\[
\Sigma \subset N_0 \cup \left( \cup_{j\in N}\ \Lambda_j\right),
\qquad\mbox{ and } \calH^1(N_0) = 0,
\]
and every point in $\Sigma \setminus N_0$ is contained in exactly one $\Lambda_j$. We then define, for $x\in \Sigma \setminus N_0$ 
\[
\mbox{ap} T_x\Sigma = T_x\Lambda_j\mbox{ for the unique $j$ such that }x\in \Lambda_j.
\]
We will write $\tau_\Sigma(x)$ to denote a unit vector in $\mbox{ap} T_x\Sigma$.

First, we recall that if $\cals$ is a countably $1$-rectifiable, $\calH^{1}$ measurable subset of $M$ and $\Theta:\cals\to(0,\infty)$ is a locally
$\calH^{1}$-integrable function on $\cals$ then we can use the pair $(\cals,\Theta)$ to form the measure $\calH^{1}\rest\Theta$, where we have extended
$\Theta$ to be zero outside of $\cals$.
We refer to such a measure as a \emph{rectifiable $1$-varifold}.
We also refer to the function $\Theta$ as the \emph{multiplicity function} of this rectifiable $1$-varifold and, at times,  we will write $\Theta_{\cals}$ to emphasize the association. We will also sometimes use the alternate notation $\Theta \calH^1\rest \cals$ for $\calH^1\rest \Theta_\cals$.
If $\Theta$ happens to be integer-valued $\calH^{1}$-almost everywhere then we will say this rectifiable varifold is of
\emph{integer multiplicity}.
Finally, if there is a $\lambda>0$ such that $\Theta\ge\lambda$ at $\calH^{1}$-almost every point then we say that the rectifiable varifold has
 \emph{density bounded below}.
A particular example of an integer multiplicity rectifiable $1$-varifold that we will be interested in will be integration over a countable collection
of geodesics.

\vskip.1in

Next, for a smooth Riemannian manifold, $M$, we let $PM$ be the bundle whose fiber $P_a M$ at $a\in{}M$ consists of the
lines through the
origin in $T_{a}M$.
If $x\in M$ and $\xi$ is a unit vector in $T_xM$, we will sometimes abuse notation slightly and write $(x,\xi)$
to denote the element of $PM$
\beq\label{repPM}
(x,\xi) \sim \{ s\xi : s\in \R\} \subset T_xM.
\eeq
Thus $(x,\xi)$ and $(x,-\xi)$ correspond to the same element of $PM$. Suppose that $\eta$ is a smooth
function on $M$. When representing points in $PM$ as described above, a
mapping such as $(x,\xi)\in PM\mapsto |\nabla_\xi \eta(x)|^2$ is well-defined as a function
$PM\to \R$, since it is independent of the choice of sign
for the unit vector $\xi$.

We let $\pi:PM\to{}M$ be the bundle projection.

We refer to a measure $\calV\in\calM(PM)$ as a \emph{$1$-varifold}.

Observe that to a rectifiable $1$-varifold $V = \calH^1\rest \Theta_\cals$, we may associate a $1$-varifold $\calV$
defined by
\beq\label{VvscalV}
\calV(A) := V\left(\{ a\in M : \mbox{ap}T_a \cals \in A  \} \right)
=\int_{\{a\in \cals \,:\, {\rm{ap}}T_a \cals \in A  \} } \Theta_\cals(a) \, d\calH^1.
\eeq
Roughly speaking, the difference between a rectifiable $1$-varifold and the associated general $1$-varifold is that the latter explicitly records information about the approximate tangent spaces to the set $\cals$ on which the former lives.

\subsection{Definitions: first variation, stationarity, Brakke flow}\label{ss:Brakke}


For a rectifiable 1-varifold $\nu$ given by  $\nu =\calH^1 \rest \Theta_\Sigma$, where $\Sigma$ is a 1-rectifiable set, 
the {\em first variation} of $\nu$ is a distribution, denoted $\delta\nu$, whose action on smooth vector fields $X$ is defined by
\beq\label{first.var}
\delta\nu(X) := \int_{\Sigma} ( \tau_\Sigma(x) , \nabla_{\tau_\Sigma(x)} X (x))_g \  \Theta(x) \ d\calH^{1}.
\eeq
(Note that since $\tau_\Sigma$ appears quadratically, the choice of unit vector in $\mbox{ap}T_x\Sigma$ does not matter.)
A $1$-d varifold $\nu$ of the given form is {\em stationary} if 
\beq\label{stationary.def}
\delta\nu = 0.
\eeq
We remark that in light of \eqref{geoODE}, of course it follows from an integration by parts that one can associate a multiplicity-one stationary varifold with the geodesic $\gamma$.
Properties of stationary varifolds will be recalled later as needed.

For simplicity, we discuss Brakke flows and related notions from geometric measure theory only
in the case of $1$-dimensional varifolds in the $3$-dimensional manifold $(M,g)$.

Let
\beq\label{nustar}
\nu_*^t =\Theta_*(x,t) \calH^{1}\rest \Sigma^t_\nu, \qquad t\ge 0
\eeq
be a family of rectifiable $1$-varifolds in $M$.
To say that $(\nu_*^t)_{t>0}$ is a {\it Brakke flow} means that for $t>0$ there exists a 
$\nu^t_*$-integrable vector field $H(\cdot,t)$ along $\Sigma^t_\nu$ 
(that is, $\vec H(x,t)\in T_xM$ for $\nu^t_*$ almost every $x\in \Sigma^t_\nu$)
such that  the following hold.
First,
\[
\delta\nu^t_*(X) = \int_{M} (X ,  H)_g \, d\nu^t_* = \int_{\Sigma^t_\nu}(X ,  H)_g \, \Theta_*(x,t)\, d\calH^{1}
\]
for all $C^1$ vector fields $X$.
Second, for every $t >0$ and every nonnegative $\chi\in C^2(M)$,
\beq\label{Brakke.ineq}
\limsup_{s\to t}\frac{\nu^t_*(\chi) - \nu^s_*(\chi)}{t-s}
\le  - \int_M  \chi | H|^2 d\nu^t_* + \int_M (\nabla\chi, P^\perp (H ))_g d\nu^t_*
\eeq
where at a point $x\in \Sigma^t_\nu$ at which $T_x\Sigma^t_\nu$ exists,
we write $P^\perp (\cdot)$ to denote orthogonal projection onto $(T_x\Sigma^t_\nu)^\perp\subset T_xM$.

For  $(\nu^t_*)_{t>0}$ a Brakke flow in $M$ of the form \eqref{nustar},
it is an immediate consequence of \eqref{Brakke.ineq}
that 
\beq\label{nut.decr}
t\mapsto \nu^t_*(M)\mbox{ is nonincreasing}.
\eeq
Another simple fact we will need is the following.
\begin{lemma}
If there exist numbers $0\le a<b$ such that 
\beq\label{nut.const}
t\mapsto \nu^t_*(M)\mbox{ is constant for $a<t<b$}
\eeq
then
\beq
\label{nut.stationary}
\mbox{$\exists$ a stationary varifold $V_*$ in $M$ such that $\nu^t_* =  V_*$ for all $a<t<b$.} 
\eeq
\end{lemma}

\begin{proof}

Clearly, if \eqref{nut.const} holds, then by taking $\chi = 1$ in \eqref{Brakke.ineq}, we find that $ H = 0$ a.e. in $\Sigma^t_\nu$ for every $t\in (a,b)$.
It follows that
$\nu^t_*$ is stationary for such $t$.
It is also easy to see that $t\mapsto \nu^t_*$ is constant for $t\in (a,b)$.
Indeed, given any nonnegative $\chi_1\in C^2(M)$, choose $\chi_2\in C^2(M)$ such that 
$\chi_1+\chi_2$ is constant on $M$. Then it
follows from \eqref{nut.const} that
\[
\int_M(\chi_1+\chi_2)\,d\nu^t_* = \int_M \chi_1 \,d\nu^t_* + \int_M \chi_2\, d\nu^t_*= c \nu^t_*(M) \mbox{ is constant }
\mbox{ for }t\in (a,b).
\]
On the other hand, since $ H = 0$, it follows from \eqref{Brakke.ineq} that 
\[
  \int_M \chi_j \,d\nu^t_* \mbox{ is nonincreasing  for $j=1,2$, for $t\in (a,b)$.}
\]
These together imply that $t\mapsto  \int_M \chi_j \,d\nu^t_*$ is constant for $j=1,2$.
Since this holds for all nonnegative $\chi_1\in C^2(M)$, it easily follows that
$\nu^t_*$ does not depend on $t\in (a,b)$, proving \eqref{nut.stationary}.
\end{proof}

We will make heavy use of results from a paper of Allard and Almgren \cite{AA} on stationary $1$-dimensional varifolds with positive density in a Riemannian manifold. Among other results, they prove  that a stationary $1$-d varifold  with density bounded away from $0$ is supported on a finite or countable union of geodesic segments terminating at singular points. From these singular points multiple segments emanate, with a balance condition on the weighted sum, at each singular point, of the tangent vectors generating the geodesics that meet there. Other results from \cite{AA} will be cited as the need arises.

\subsection{Asymptotic analysis of the Ginzburg-Landau heat flow }

As a last preliminary, we state a recently established extension to the Riemannian setting of
a theorem of  Bethuel, Orlandi and Smets \cite{BOS}, who built on prior work of a number of authors, including \cite{Ilmanen, AmbrosioSoner, LinRiv3}.

The theorem quoted below is  proved in \cite{Colinet}.

\begin{theorem}\label{T:BOS}Assume that $(N,h)$ is a closed Riemannian manifold of dimension $n\ge 3$.
Let $u_\ep: N\times [0,\infty)\to \C$ solve the Ginzburg-Landau heat flow
\[
\pp_t u_\ep - \Delta u_\ep + \frac 1{\ep^2}(|u_\ep|^2-1)u_\ep = 0 \qquad\mbox{ on }N\times(0,\infty)
\]
with initial data $u_\ep(x,0)= u_\ep^0(x)$. Assume that there exists  $M_0>0$ such that 
\[
\EUep(u_\ep^0) \le M_0.
\]
For every $t\ge 0$, let $\mu_\ep^t$ be the measure on $N$ defined by
\[
\mu_\ep^t(A) = \int_A \frac{e_\ep(u_\ep(\cdot, t))}{\logeps}\, \emph{\vol}
 \qquad\mbox{ for every Borel }
A\subset N.
\]
Then after passing to a subsequence (still denoted simply by $\ep$),
there exist measures
$\mu_*^t\in\calM(N)$ for every $t>0$ such that
\[
\begin{aligned}
\mu_\ep^t \rightharpoonup \mu_*^t \quad &\mbox{weakly as measures for every }t > 0.
\end{aligned}
\]

Moreover, there exists a smooth function $\Phi_*$ solving the heat equation on $N
\times(0,\infty)$  and a family of measures $(\nu^t_*)_{t>0}$ on $N$, such that for every $t> 0$
\[
\mu_*^t =  \frac 12 |\nabla \Phi_*(x,t)|^2\,\emph{\vol} + \nu^t_*
\]
with $\nu^t_*$ taking the form
\beq\label{nustar.form}
\nu_*^t =\Theta_*(x,t) \calH^{n-2}\rest \Sigma^t_\nu
\eeq
where $\Sigma_\nu^t$ is an $(n-2)$-dimensional rectifiable subset of $N$ and $\Theta_*$
is a bounded measurable function. In addition, there exists a 
function $\eta:(0,\infty)\to (0,\infty)$ such that 
\beq\label{ldb}
\Theta_*(x,t)  = \lim_{r\to 0} \frac {\nu^t_*(B_{r}(x))}{\omega_n r^{n-2}} \ge \eta(t) 
\eeq
for $\calH^{n-2}$ a.e.  $x\in \Sigma^t_\nu$, for a.e. $t>0$.
Finally, the family $(\nu_*^t)_{t>0}$ is a Brakke flow.
\end{theorem}

\section{A non-existence result for stationary $1$-varifolds near a non-degenerate geodesic}

The proof of our main result hinges crucially on showing there is no stationary varifold sitting over a $1$-current that is nearby $T_{\gamma}$, the $1$-current associated with the non-degenerate geodesic $\gamma$. While the non-degeneracy assumption \eqref{gamma.def} easily precludes the existence of another nearby smooth geodesic, it is the need to rule out proximity in the weaker sense of \eqref{rflat}, \eqref{Vstarontop} below and within the larger class of varifolds that makes the result below much more challenging to establish.

\begin{proposition}
\label{lem:noVstar}
Let $T_\gamma$ be the $1$-current  in $M$ corresponding to integration
over the nondegenerate geodesic  $\gamma$, and let $\eta>0$ be given. Then there exists $r_0>0$ depending on $M, \gamma$, and $ \eta$,
such that for $0<r<r_0$ there is no stationary $1$-dimensional rectifiable varifold $V_*$ and  $1$-current $J_1\in \calR_1 \cap \calF_1'(M)$ satisfying the conditions
\beq\label{Vstar}
V_* = \Theta_*(x)\calH^1\rest \Sigma_{*},
\eeq
for $\Sigma_{*}\;1$-rectifiable and $\Theta_*\ge \eta>0\;\calH^1$  
 a.e.  in $\Sigma_{*}$,
 \vskip.05in
\beq
\| J_1 - T_\gamma\|_\calF=r,\label{rflat}
\eeq
and
\beq\label{Vstarontop}
 \qquad V_* \ge \| J_1\|, \qquad V_*(M)\le L.
\eeq
\end{proposition}

The starting point of the proof is provided by the
following lemma, established by Mesaric \cite{Mesaric}.

\begin{lemma}\label{meslem}
For $T_\gamma$ as above, let $J_1\in \calR_1\cap \calF_1'(M)$ be a current satisfying \eqref{rflat}, and such
that 
\[
\partial J_1 = 0\qquad\mbox{and } \qquad
\mass(J_1)\le L.
\]
Then provided $r$ is taken sufficiently small, there is a $1$-current $J_{1}^{*}\in\calR_{1}(M)$ such that the support of $J_{1}^{*}$, denoted by $\Gamma^{*}$, consists of a single Lipschitz curve with no boundary satisfying
\beq
\Gamma^{*}\subset{}K_{2\sqrt[3]{r}}\cap\spt(J_{1}),
\eeq
and
\beq\label{EachHeight}
\Gamma^{*}\cap\tau^{-1}(t)\neq\varnothing\text{ for all }t\in\R\slash{}L\Z,
\eeq
(cf. \eqref{y.tau.def} for the definition of $\tau$).\\ 

In addition,
\beq
\mass(J_{1}-J_{1}^{*})=\mass(J_{1})-\mass(J_{1}^{*}),
\eeq
there exists a constant $C_{1}>0$ such that
\beq\label{C1}
\mass(J_{1}^{*})\ge L-C_{1}\sqrt[3]{r},
\eeq
and
\begin{align}
J_{1}^{*}-&T_{\gamma}=\partial{}S^{*}\text{ for some }2\text{-current }S^{*}\text{ with}\\
&\spt(S^{*})\subset\overline{K_{5\sqrt[3]{r}}}\text{ and }\mass(S^{*})<\infty.  \nonumber
\end{align}
\end{lemma}

This is demonstrated in Lemma $4.4$ and comments following Lemma $4.6$ of \cite{Mesaric}. The proof is an adaptation to the Riemannian setting of arguments from \cite{JSt}, Lemma 5.5.
The idea is that \eqref{rflat} and the definition of the flat norm
imply that a large set of transverse slices to $\Gamma$ must intersect
$J_1$ at exactly one point, and this point must be close to $\Gamma$. Behavior of $J_1$ on other slices is constrained by the assumption that $\partial J_1=0$ and the mass bound. The proof also uses Federer's decomposition of integral $1$-currents, \cite{Federer}, 4.2.25.

\vskip.1in
\noindent{{\it Proof of Proposition \ref{lem:noVstar}}\\
\noindent
{\em{Step 1:}}
First we show that the rectifiable varifold $V_{*}$ does not have any mass outside of $K_{\sqrt[4]{r}}$.
The idea is that if this fails, then the monotonicity formula and \eqref{C1} would contradict the assumption
$\mass(V_*) \le L$. This argument relies crucially on the uniform lower density bound for $V_*$.\newline
\newline
We recall that the Hessian Comparison Theorem, see Theorem $6.6.1$ of \cite{Jos}, gives us that
if $\mu>0$ is an upper bound on the absolute value of the scalar curvature over $M$ and $r>0$ is chosen so that
\[
r<\frac{1}{2}\min\left\{r_{0},\frac{\pi}{2\sqrt{\mu}}\right\},
\]
then for each $p\in{}M$ and all $x\in{}B_{r}(p)$ and $v\in{}T_{x}M$ we have
\beq \label{HessComp}
\sqrt{\mu}\rho(x)\cot\left(\sqrt{\mu}\rho(x)\right)\left|v\right|^{2}\le\text{Hess}_{g}\left(\frac{\rho^{2}}{2}\right)(v,v)
\le\sqrt{\mu}\rho(x)\coth\left(\sqrt{\mu}\rho(x)\right)\left|v\right|^{2}
\eeq
where $\rho(x)=d(x,p)$.
In view of the lower density bound $\Theta_* \ge \eta$,  $V_*$ almost everywhere, it follows from
a Riemannian version of the Monotonicity Formula (established with different notation in \cite{AA} in item $(5)$ of Theorem $1$ on pg. $87$)
that there exists $r_{con}(M)>0$ such that for every $0<s<r_{con}$,
\beq
\eta\le\frac{1}{2s}\int_{B_{s}(p)}\!{}\text{Hess}_{g}\left(\frac{\rho^{2}(x)}{2}\right)(v,v)dV_{*}
\eeq
for $V_*$-almost every $p\in M$, where
$v$ is a unit vector in $\mbox{ap} T_x \Sigma_*$. (Clearly the value of $\text{Hess}_{g}\left(\frac{\rho^{2}(x)}{2}\right)(v,v)$
does not depend on which unit tangent is chosen.)

It follows from $\eqref{HessComp}$ that
\[
2s\eta\le(1+\mu{}s^{2})V_{*}\left(B_{s}(p)\right)
\]
for all $0<s<\frac{1}{2}\min\left\{r_{0},r_{con},\frac{\pi}{2\sqrt{\mu}}\right\}$ and $V_*$-almost every $p\in M$.
We conclude that for all $0<s<\frac{1}{2}\min\left\{r_{0},r_{con},\frac{\pi}{2\sqrt{\mu}},1\right\}$ we have
\beq\label{LowerBoundBall}
V_{*}\left(B_{s}(p)\right)\ge\frac{2s\eta}{1+\mu}, \qquad \mbox{ for } V_* \mbox{ {\em a.e. }}p\in M.
\eeq

We now use this to prove that 
if  $r$ is chosen sufficiently small, then
\begin{equation}
V_{*}\left(M\setminus{}K_{\sqrt[4]{r}} \right)=0.\label{none}
\end{equation}
To verify \eqref{none}, suppose to the contrary; then 
$V_{*}\left(M\setminus{}K_{\sqrt[4]{r}} \right)>0$
 and so there is a point $p$ of
$\spt(V_{*})$ in $M\setminus{}K_{\sqrt[4]{r}}$ for which $\eqref{LowerBoundBall}$ holds.
By $\eqref{LowerBoundBall}$ we have that
\[
V_{*}\left(B_{\frac 12\sqrt[4]{r}} (p)\right)\ge \frac{\eta}{1+\mu} \sqrt[4]{r},
\quad\qquad\mbox{ if }r<\min\left\{r_{con},r_{0},\frac{\pi}{2\sqrt{\mu}},1\right\}^{4}.
\]
By shrinking $r_0$ if necessary, we may  assume that
$B_{\frac 12 \sqrt[4]{r}}(p)\cap K_{2\sqrt[3]{r}}=\varnothing$.
Hence, appealing to \eqref{C1}, we find that
\[
L\ge{}V_{*}(M)\ge{}V_{*}(K_{2\sqrt[3]{r}})+
V_{*}\left(B_{\frac 12\sqrt[4]{r}} (p)\right)
>L-C_{1}\sqrt[3]{r}+\frac{\eta}{1+\mu} \sqrt[4]{r}.
\]
Choosing $r$ smaller if necessary,  depending on $C_1,\eta, \mu$, this yields a contradiction.
We conclude that \eqref{none} holds.
Since $\left\|J_{1}\right\|\le{}V_{*}$, we remark that  $J_{1}$ is also supported in
$K_{\sqrt[4]{r}}$.
\newline

\vskip.1in
{\em{Step 2:}}
Next we demonstrate that
\beq\label{LeadingOrderTau}
\left|\nabla\tau(x)\right|=1+O(\sqrt[4]{r}),\qquad
\left[\text{Hess}_{g}(\tau)(x)\right]\left(v,v\right)=O(\sqrt[4]{r})
\eeq
where $v\in{}T_{x}M$ is a unit vector, $x\in{}K_{\sqrt[4]{r}}$, and $\tau$ is the
mapping defined in \eqref{y.tau.def}.\newline
\newline
We prove only the statement about the Hessian, as the gradient estimate follows by similar
arguments.
In coordinates introduced by $\psi:B_{\sqrt[4]{r}}(0)\times(0,L)\to{}K_{\sqrt[4]{r}}$  as defined in
\eqref{TubularCoordinates}, we can write $\tau$ as
\[ 
\tilde{\tau}(y,t)=t
\] 
where $\tilde{\tau}=\tau\circ\psi$. 
These are what are called Fermi coordinates, and a basic fact, proved for example in
Section V of \cite{MM},
is that
the vectors $\Xi_1,\Xi_2$ in \eqref{TubularCoordinates} can be chosen so that
all Christoffel symbols vanish along the central geodesic, that is, when $y=0$:
\[
\Gamma^k_{ij}(0,t) = 0 \quad\mbox{ for }i,j,k=1,\ldots, 3
\]
where  $x^{k}=y^{k}$ for $k=1,2$, and $x^{3}=t$. 
In general, the expression for the Hessian in coordinates is
\[ 
\text{Hess}_{g}(\tau)(\psi(0,t))=\sum_{i=1}^{3}\sum_{j=1}^{3}
\left(\frac{\partial^{2}\tilde{\tau}}{\partial{}x^{i}\partial{}x^{j}}(0,t)-\sum_{k=1}^{3}\Gamma_{ij}^{k}(0,t)
\frac{\partial\tilde{\tau}}{\partial{}x^{k}}(0,t)\right)\mathrm{d}x^{i}\otimes\mathrm{d}x^{j},
\] 
see for example \cite{Jos}, Definition 4.3.5.
By combining these, we readily deduce that 
\[
\text{Hess}_{g}(\tau)(\psi(0,t))=0
\]
and thus
that $\text{Hess}_{g}(\tau)(\psi(y,t))_{ij}=O_{t}(\left|y\right|)$
for $1\le i,j\le 3$. The Hessian estimate in \eqref{LeadingOrderTau} follows directly.
\newline

\vskip.1in
{\em{Step 3:}}
Let $\calV_*$ be the $1$-varifold associated as in \eqref{VvscalV} to the rectifiable $1$-varifold $V_*$.
We next demonstrate that for each $\delta>0$ there is $r_{1}>0$ such that if $0<r<r_{1}$ in \eqref{rflat} and hence in \eqref{none}, then

\beq
\calV_{*}\left( \left\{ (x,\xi)\in{}PM: |\nabla_\xi\tau(x)|^2  \le (1-\delta)^2  \right\}\right)<\delta ,
\label{bigangle}
\eeq
where we recall our convention that a generic element of $PM$  --- that is, a line in $T_xM$ for some $x\in M$ --- is represented by a pair $(x,\xi)$, where $\xi$ is a unit vector
in $T_xM$ spanning the given line, see \eqref{repPM}. 
This will establish that most tangent vectors to the support of $V_{*}$ are, according to the measure $\calV_{*}$, nearly parallel to $\nabla\tau$.

We suppose toward a contradiction that there is a $\delta>0$, a sequence $\left(r_{k}\right)_{k\in\N}$ tending to $0$ from the right, and a sequence of
stationary rectifiable varifolds $\left(V_{k}\right)_{k\in\N}$ on $M$ satisfying the hypotheses of Proposition \ref{lem:noVstar} with $r$ replaced by $r_k$ in \eqref{rflat},
and 
such that the associated
$1$-varifolds $\calV_k$ satisfy}
\beq\label{convk}
\calV_{k}\left( \left\{ (x,\xi)\in{}PM: |\nabla_\xi\tau(x)|^2  \le (1-\delta)^2  \right\}\right)\ge \delta ,
\eeq
for all $k\in\N$. In particular we have
\beq\label{VarifoldSequenceProperties}
V_{k}(M)\le{}L,\hspace{10pt}
\spt(V_{k})\subset{}K_{\sqrt[4]{r_{k}}},\hspace{10pt}
\text{and }\;\Theta_{V_k}(x)\ge\eta\text{ for }x\in\spt(V_{k}).
\eeq
Since $\left(\calV_{k}\right)_{k\in\N}$ is a sequence of stationary rectifiable varifolds we may combine $\eqref{VarifoldSequenceProperties}$ with the compactness result of Theorem $42.7$, pg. 247 of \cite{LS} to conclude that there is a subsequence $\left(\calV_{k_{j}}\right)_{j\in\N}$ and a rectifiable varifold $V$ with associated multiplicity $\Theta_V$
such that
\begin{enumerate}
\item\label{WeakConvergence}
$\calV_{k_{j}}\rightharpoonup\calV$ weakly as measures
\item\label{LimitLowerBound}
$\Theta_V(x)\ge\eta$ on $\spt(V)$
\item\label{VarifoldSemicontinuity}
$\delta{}V(W)\le\liminf_{j\to\infty}\delta{}V_{k_{j}}(W)$ for all $W\subset\subset{}M$
\end{enumerate}
It follows from $\eqref{VarifoldSemicontinuity}$ and the fact that each $\calV_{k_{j}}$ is stationary that $\calV$ is also stationary.
Then from $\eqref{WeakConvergence}$ and the fact that $\spt(V_{k_{j}})\subset{}K_{\sqrt[4]{r_{k_{j}}}}$ we conclude that $\spt(V)\subset\Gamma$.
Next, we observe that, due to $\eqref{EachHeight}$, each $V_{k_{j}}$ has support that contains a closed curve that meets every level set of $\tau$.
Hence, $\spt(V)=\Gamma$ as a result of $\eqref{WeakConvergence}$.
Observe that since $\calV$ is a stationary varifold with density bounded below and $\spt(V)=\Gamma$, then by the Theorem on page $89$ of \cite{AA} we have that $V$ is simply a constant multiplicity multiple of the stationary varifold $V_{\Gamma}$ associated with the geodesic $\gamma$. 
Applying the weak convergence \eqref{WeakConvergence} to \eqref{convk}, however, we see that
\[
\calV\left( \left\{ (x,\xi)\in{}PM: |\nabla_\xi\tau(x)|^2  \le (1-\delta)^2  \right\}\right)\ge\delta ,
\]
an impossibility given that all tangent vectors $\xi$ along $\Gamma$ coincide with $\pm \nabla \tau$.
\vskip.1in
\noindent{\em{Step 4:}}
Next we introduce three sets corresponding to slices normal to the central geodesic that are in some sense bad. We will argue that two of them correspond to sets of $t$-values of measure zero while the third is of small measure.

We introduce the first such set,  $\calB_1$, through the function $h:(0,L)\to\R$ given by
\[
h(t)=\int_{\pi^{-1}\left(\left\{0\le\tau\le{}t\right\}\cap{}K_{\sqrt[4]{r}}\right)}\!{}\left|\nabla_{\xi}\tau(x)\right|^{2}d\calV_{*}(x,\xi),
\]
with  $\calB_1$ defined by
\beq
\calB_{1}:=\left\{a\in(0,L):\;h\;\text{is not differentiable at}\;a\right\}.\label{B0}
\eeq
Since $h$ is non-decreasing, it is differentiable $\calL^{1}$-almost everywhere and consequently $\calL^{1}(\calB_{1})=0$.

Now we recall that the singular set $S_{V_{*}}$, as defined in \cite{AA}, is the set of points of $M$ near which $\Theta_{V_{*}}$, restricted
to $\spt(V_{*})$, is not constant.
Then we introduce the set $\calB_{2}$ as the set of slices meeting the singular set:
\[
\calB_{2}:=\left\{t\in(0,L):\left\{\tau=t\right\}\cap{}S_{V_{*}}\neq\varnothing\right\}
 =  \{ \tau(x) : x\in S_{V_*}\cap K_{\sqrt[4]{r}} \} .
\]
We claim that $
\calL^{1}\left(\calB_{2}\right)=0
$
as well.

To this end, we note that in the remark following the theorem on page $89$ of \cite{AA}, it is stated that
\beq\label{VarifoldMeasureSingularSet}
V_{*}\left(S_{V_{*}}\right)=0,
\eeq
and so by $\eqref{ldb}$ and $\eqref{VarifoldMeasureSingularSet}$ we have that
\[
0=\int_{S_{V_{*}}}\!{}\Theta(x)d\calH^{1}(x)\ge\eta\,\calH^{1}\left(S_{V_{*}}\right).
\]
We conclude that
\beq\label{HausdorffMeasureSingularSet}
\calH^{1}\left(S_{V_{*}}\right)=0.
\eeq
Since $\calB_2$ is the image of a subset of $S_{V_*}$ by the Lipschitz map $\tau:K_{\sqrt[4]r}\to (0,L)$,
it follows that $\calL^{1}\left(\calB_{2}\right)=0$ as claimed.

The final `bad' set of slices is $\calB_{3}$ defined by
\[
\calB_{3}:=\left\{t\in(0,L):\;\exists
(x,\xi)\in\spt(\calV_{*})\;s.t.\;|\nabla_\xi\tau(x)|^2<(1-\delta)^2,\tau(x)=t\right\}.
\]
Replacing the role of the equality \eqref{VarifoldMeasureSingularSet} by the inequality \eqref{bigangle} in the argument above, the same line of reasoning goes to show that there is a constant $C_{2}>0$ such that
\[
\calL^{1}\left(\calB_{3}\right)<\frac{C_{2}\delta}{\eta},
\]
where $r$ is chosen sufficiently small.
\newline

{\em{Step 5:}}
We now use the results obtained in Steps $1$ and $2$ to show that, for $r$ and $\delta$ chosen sufficiently small, and
$a,b\in(0,L)\setminus\calB_{1}$ we have
\[
h'(b)=h'(a)+O(\sqrt[4]{r})
\]
where $h$ is as defined in Step 4.
\newline
\newline
For $s$ small and positive, we define the function
\[
H_{a,b;s}(t)=
\begin{cases}
1&\text{if }a+s<t<b-s\\
0&\text{if }0<t<a\text{ or }b<t<L\\
\frac{t-a}{s}&\text{if }a\le{}t\le{}a+s\\
\frac{b-t}{s}&\text{if }b-s\le{}t\le{}b,
\end{cases}
\]
and we let $X$ be a smooth vector field on $M$ such that $X(x) = H_{a,b;s}(\tau(x))\nabla\tau(x)$ for $x\in K_{r_{0}\slash2}$.
The fact that $\calV_*$ is stationary implies that $\delta V_*(X)=0$, cf. \eqref{first.var},
and this means that
\begin{eqnarray}
\int_{\pi^{-1}\left(K_{\sqrt[4]{r}}\right)}&&\!{}H_{a,b;s}'(\tau)\left|\nabla_{\xi}\tau\right|^{2}d\calV_{*}(x,\xi)
+\nonumber\\
&&\int_{\pi^{-1}\left(K_{\sqrt[4]{r}}\right)}\!{}H_{a,b;s}(\tau)\text{Hess}_{g}(\tau)\left(\xi,\xi\right)d\calV_{*}(x,\xi)=0
,\label{VariationCondition}
\end{eqnarray}
We have used Step $1$ to obtain that $V_{*}$ is concentrated on $K_{\sqrt[4]{r}}$.
Next, we use Step $2$, $V_{*}(M)\le{}L$, and the fact that
$\left\|H_{a,b;s}\right\|_{L^{\infty}(\mathbb{R})}=1$ to conclude that
\beq\label{HessianBound}
\int_{\pi^{-1}\left(K_{\sqrt[4]{r}}\right)}\!{}H_{a,b;s}(\tau)\text{Hess}(\tau)\left (\xi,\xi \right)d\calV_{*}(x,\xi) =O(\sqrt[4]{r}).
\eeq
Then we observe that, by the definition of $H_{a,b;s}$, we have
\begin{align*}
\int_{\pi^{-1}\left(K_{\sqrt[4]{r}}\right)}\!{}H_{a,b;s}'\left|\nabla_{\xi}\tau\right|^{2}d\calV_{*}(x,\xi)
&=\frac{1}{s}\int_{\pi^{-1}\left(\left\{a\le\tau\le{}a+s\right\}\cap{}K_{\sqrt[4]{r}}\right)}\!{}\left|\nabla_{\xi}\tau\right|^{2}d\calV_{*}(x,\xi)\\
&-\frac{1}{s}\int_{\pi^{-1}\left(\left\{b-s\le\tau\le{}b\right\}\cap{}K_{\sqrt[4]{r}}\right)}\!{}\left|\nabla_{\xi}\tau\right|^{2}d\calV_{*}(x,\xi).
\end{align*}

Combing this with $\eqref{VariationCondition}$ and \eqref{HessianBound} yields
\[
\frac{h(b)-h(b-s)}{s}=\frac{h(a+s)-h(a)}{s}+O(\sqrt[4]{r}).
\]
Letting $s$ tend to zero, we find that
\beq
h'(b)=h'(a)+O(\sqrt[4]{r}).\label{hbprime}
\eeq
\newline
{\em{Step 6:}} 
We next introduce a set of `good' slices via
\[
\calG:=\left\{a\in(0,L):\calH^{0}\left(\left\{\tau=a\right\}\cap\spt(V_{*})\right)=1\right\},
\]
and in this step we will demonstrate that
\beq\label{SingleIntersection}
(0,L)\setminus\left(\calB_{1}\cup\calB_{2}\cup\calB_{3}\right)\subset\calG.
\eeq
From this and Step 4 it will follow that
\beq\label{GLargeMeasure}
\calL^{1}\left(\calG\right)\ge{}L-\frac{C_{2}\delta}{\eta}
\eeq
provided that $r$ and $\delta$ are chosen sufficiently small.
\newline
\newline
Suppose by way of contradiction that there exists a value $a$ such that
\beq
a\in (0,L)\setminus\left(\calB_{1}\cup\calB_{2}\cup\calB_{3}\cup \calG\right).\label{contraa}
\eeq

Then in light of \eqref{EachHeight} we have that
\[
\calH^{0}\left(\left\{\tau=a\right\}\cap\spt(V_{*})\right)\ge2
\]
for some $a\in(0,L)\setminus\left(\calB_{1}\cup\calB_{2}\cup\calB_{3}\right)$.
We first argue that for such an $a$ and every $c,\delta\in (0,1)$ it must hold that
\beq
h'(a)\ge(1-c)(1-\delta)^{2}(1+\eta)\label{hprime}
\eeq
provided $r$ is chosen sufficiently small, depending on $c$.
Fix  $0<c,\delta <1$, and choose $1<\alpha<\frac{1}{\delta}$.
Note that since $a\not\in\calB_{2}$ then by item $(3)$ of the theorem on page $89$ of \cite{AA} we have that $\left\{\tau=a\right\}\cap\spt(V_{*})$
consists only of interior points of the constituent geodesics (or ``intervals" as they are referred to in \cite{AA}) making up $\calV_{*}$. Moreover, the endpoints of these geodesics cannot accumulate at $a$.
We conclude by compactness of $\spt(V_{*})$ that $\left\{\tau=a\right\}\cap\spt(V_{*})$ can intersect only finitely many of these constituents of
$\spt(V_{*})$.
Also since $a\not\in\calB_{3}$, it follows that this slice must intersect $\spt(V_*)$ transversally, so that $\left\{\tau=a\right\}\cap\spt(V_{*})$ consists of finitely many points, say $x_1,\ldots, x_K$, where $K\ge2$ by the choice of $a$.

It follows from the gradient estimate $|\nabla \tau|=1+O(\sqrt[4]r)$ in $K_{\sqrt[4]r}$, established in  \eqref{LeadingOrderTau}, that for $0<\alpha<\beta <L$,
the geodesic distance between the sets $\{ \tau = \alpha\}\cap K_{\sqrt[4]r}$ and $\{\tau = \beta\} \cap K_{\sqrt[4]r}$
is at least $(\beta-\alpha)/(1+O(\sqrt[4]r))$, if $\beta-\alpha< L/2$. Hence if $r$ is small enough, then
\[
B_{(1-c)s}(x_i) \cap  K_{\sqrt[4]r} \subset  \{x\in K_{\sqrt[4]r}: a-s <\tau(x) < a+s \} .
\]
(Here and below, we tacitly assume that $0<a-s<a+s<L$ and $s<L/2$.)

Next we again use that $a\notin \calB_3$ to choose $s_0>0$ small enough so that
\beq
\left|\nabla_{\xi}\tau(y)\right|^{2}\ge(1-\alpha\delta)^{2}
\eeq
if $y\in{}\cup_{i=1}^K B_{s_0}(x_{i})$
and $(y,\xi)\in \spt(\calV_*)$.
Combining these facts, for each $0<s\leq s_0$ we estimate
\[
\frac{h(a+s)-h(a-s)}{2s}\ge(1-\alpha\delta)^{2}\sum_{i=1}^{K}\frac{V_{*}\left(B_{(1-c)s}(x_{i})\right)}{2s}
\]
 We now apply item $(5)$ of the theorem on page $87$ of \cite{AA} and let $s\to0^{+}$, using the differentiability of $h$ at $a$ guaranteed by the assumption $a\notin \calB_1$, to find
\[
h'(a)\ge{}(1-c)(1-\alpha\delta)^{2}\sum_{i=1}^{K}\Theta_*(x_{i})\ge{}(1-c)(1-\alpha\delta)^{2}(1+\eta)
\]
Here we have used Lemma \ref{meslem} to assert that $\Gamma^{*}$ intersects each level set of $\tau$ with $\Theta_*(x)\geq 1$ for $x\in\Gamma^{*}$ by \eqref{Vstarontop}, and that $\Theta_*\ge\eta$ in general by $\eqref{formofVstar}$.
Since $\alpha>1$ was arbitrary we may let $\alpha\to1^{+}$ to obtain \eqref{hprime}.

In light of \eqref{hbprime}, it then follows from \eqref{hprime} that for any $b\notin\calB_1$
 we obtain
\[
h'(b)\ge{}(1-c)(1-\delta)^{2}(1+\eta)+O(\sqrt[4]{r})
\]
Thus, choosing $c$, $r$, and $\delta$ sufficiently small and recalling that $\calL^1(\calB_1)=0$,
we deduce that
\beq
h'(b)\ge1+\frac{\eta}{2}\qquad\mbox{ for {\em a.e.} }b\in (0,L).
\eeq
Thus, if there were a value $a\in(0,L)$ satisfying \eqref{contraa}, then
\begin{eqnarray*}
&&\left[1+C_3\sqrt[4]{r}\right]L\ge{}
\left[1+C_3\sqrt[4]{r}\right]V_*\left(\pi^{-1}\left\{0\le\tau\le{}L\right\}\right)\\
&&\geq h(L)-h(0)
=\int_{0}^{L}\!{}h'(s)ds
\ge\left(1+\frac{\eta}{2}\right)L.
\end{eqnarray*}
Here we use \eqref{LeadingOrderTau} in the second inequality and the constant $C_3$ depends only on $M$ and $\Gamma$.
If we choose $r$ sufficiently small, the contradiction is reached,
establishing \eqref{SingleIntersection} and \eqref{GLargeMeasure}.
\vskip.1in
\noindent
{\em{Step 7:}}
In this step we show that 
\[
S_{V_{*}}=\varnothing.
\]
It will immediately follow that the Lipschitz curve $\Gamma^{*}$ guaranteed by Lemma \ref{meslem} represents the entire rectifiable varifold and is in fact a closed geodesic.

Crucial use in this step will be made of the following general property of stationary $1$-varifolds (cf. \cite{AA}, pg. 88.):
\vskip.1in

{\it Every point $p\in M$ is contained in an open set $U_p$ such that if $V$ is any stationary varifold on $M$ with support in $U_p$, and if the support of $\delta V$ consists of exactly two points, then
$V$ is the varifold corresponding to a constant multiple of the geodesic joining these two points.}
\vskip.1in
This result is proved in \cite{AA}
for possibly noncompact manifolds. Since $M$ is compact, we may invoke
the Lebesgue Number Lemma to conclude that there exists $\lambda>0$ such that for any $p\in M$, the geodesic ball $B_\lambda(p)$ has the stated property.
\vskip.1in
For  any $A\subset (0,L)$, we will write
\[
K_{\sqrt[4]r}(A) := \{ x\in K_{\sqrt[4]r} : \tau(x)\in A \}  = \{ \psi(y,t): |y| \le \sqrt[4]r,  \ \ t\in A \}.
\]
By extending $\psi$ to be periodic with respect to the $t$ variable in the natural way, we
can define $K_{\sqrt[4]r}(A)$ for any $A\subset \R$.
\vskip.1in
By shrinking $r$ and $\delta$, we may arrange that if $I$ is any interval of length at most $2C_2\delta/\eta$,
where $C_2$ is the constant appearing in \eqref{GLargeMeasure}, then $K_{\sqrt[4]r} (I)$ 
is contained in a ball of radius $\lambda$.
\vskip.1in
We will prove the claim by showing that 
\[
\mbox{ for any $t\in [0,L)$, }\qquad K_{\sqrt[4]r} (\{t\})\cap S_{V_*} = \varnothing.
\]
Indeed, for any $t$ we can appeal to \eqref{GLargeMeasure}
to find some $s_1, s_2\in \calG$  such that 
\[
s_1<t<s_2, \qquad  s_2-s_1 < \frac{2C_2\delta}\eta.
\]
We now apply the result stated at
the outset of this step to the varifold 
\[
\tilde{V}=V_{*}\rest K_{\sqrt[4]r}([s_1,s_1]),
\]
in an open ball $B_\lambda(p)$ that contains $K_{\sqrt[4]r}([s_1,s_1])$.
The definition of $\calG$ implies that $V_*$ intersects $\{\tau = s_j\}$ in
exactly one point, say $x_j$, and that $\delta \tilde V$ is supported in $\{x_1,x_2\}$. 
Hence, this restriction of $V_*$ consists of a multiple of the geodesic joining these two points. This immediately
implies the claim.
\vskip.1in
Since $S_{V_*}=\varnothing$, $V_*$ must simply be the rectifiable varifold associated with a single, closed, smooth geodesic.

By perhaps shrinking $r$ one more time and applying the Morse-Palais Lemma, cf. \cite{Pal}, pg. 307, we may conclude that the central geodesic $\Gamma$, being a nondegenerate critical point of length, 
is isolated and so necessarily $\spt(V_{*})=\Gamma$.
However, this contradicts $\eqref{J0.p1}$ since $r>0$, and the proof of Proposition \ref{lem:noVstar} complete.

\qed

\section{Finding good trajectories}\label{sec:4}

The critical points of Ginzburg-Landau that we seek will be obtained as limits of certain carefully chosen trajectories of the Ginzburg-Landau heat flow. In this section we identify these trajectories.

\begin{proposition}\label{prop:goodtrajectory} Given $\delta>0$, there exists $\ep_0(\delta)>0$ such that for every $\ep\in (0,\ep_0)$
and every $\tau>0$, there is a solution $u_\ep = u_\ep(x,t; \delta,\tau)$ of the Ginzburg-Landau heat flow
\begin{eqnarray}
&&\partial_t u_\ep = \Delta u_\ep - \frac 1{\ep^2}(|u_\ep|^2-1)u_\ep  
\qquad\mbox{ in } M\times(0,\infty),
\label{glh}\\
&&u_\e(x,0)=u_\e^0(x)\qquad\mbox{ for } x\in M,\nonumber
\end{eqnarray}
such that $|u_\ep(x,t)|\le 1$ for every $(x,t)\in M\times [0,\infty)$, and 
\beq\label{gt.c1}
\| \frac 1 \pi \star J u_\ep^0- T_\gamma\|_\calF<\delta, \qquad
\EUep(u_\ep^0) \le L+\delta, \qquad
\EUep(u_\ep(t))  \ge L-\delta\ \mbox{ for all }t\in [0,\tau].
\eeq
\end{proposition}

The proposition follows from small modifications of 
the asymptotic minmax theory developed in \cite{JSt, Mesaric}. Indeed,
in the end the proof  amounts to this:

\begin{proof}[short proof of Proposition \ref{prop:goodtrajectory}]
This follows from arguments in \cite{JSt, Mesaric}.
\end{proof}

In the remainder of this section we expand on this, aiming to provide enough detail to convey the main ideas, to explain where we depart from \cite{JSt, Mesaric}, and to make it possible, in principle, to check the terse proof given above.

We remark that  the main difference between \cite{JSt,Mesaric}
and our present treatment is that those earlier works 
use a pseudo-gradient flow for the energy $\EUep$,
whereas we employ a small modification of the Ginzburg-Landau heat flow \eqref{glh} for similar purposes. The use of  \eqref{glh} is necessary for our approach, due to our reliance in Section \ref{sec:5} below on Theorem \ref{T:BOS}.

\subsection{Saddle point property of $E_V$}

Our assumptions about $\gamma$ imply, roughly speaking, that the ``arclength functional"
has a local minmax geometry near $\gamma$,  as reflected in  \eqref{gxi.length},  with 
respect to smooth perturbations. In particular, there is an $\ell$-parameter family of
arclength-decreasing perturbations of $\gamma$, and  arclength increases for 
sufficiently transverse smooth perturbations. Here and below, $\ell$ is the index of $\gamma$, see \eqref{index.ell}.

The result below states that the ``generalized arclength functional" $E_V$ defined in \eqref{EV.def} has a   a saddle point,
in a suitable weak sense, at the current $T_\gamma\in V$ corresponding to $\gamma$. 
The relevant notion of saddle point was first introduced in \cite {JSt}.

\begin{lemma}[cf. Theorem 4.1, \cite{Mesaric}]\label{thm:mesaric}
For the geodesic $\gamma$ satisfying \eqref{gamma.def} and \eqref{Tbounds}, the associated current $T_\gamma$ is a saddle point of $E_V$ in the sense that there exist $R,\delta_0>0$ and
continuous functions
\[
P_{WV}:V\to\R^\ell, \qquad Q_{VW}: W \to V 
\]
for
\beq\label{W.def}
W =  B_R^\ell = \{ w\in \R^\ell : |w|<R\}
\eeq
such that 
$P_{WV}(T_\gamma)=0$, and the following conditions are satisfied:
\begin{align}
&E_V(T_\gamma) < E_V(v)\mbox{ for } \{ v\in V : 0 <  \| v
-T_\gamma\|_{\calF} \le \delta_0,  P_{WV}(v) = 0\}, \label{saddlepoint1}\\
& Q_{VW}(0) =  T_\gamma,\label{0map}\\
& P_{WV}\circ Q_{VW}(w) = w\quad \mbox{ for all } w\in
W,\label{inverse}
\\
& \mbox{ and for every $r>0$,} \sup_{ \{w\in W, |w|\ge r\}}
E_V(Q_{VW}(w)) <E_V(T_\gamma). \label{saddlepoint2}
\end{align}
\end{lemma}

We sketch the proof from \cite{Mesaric}, although we note that this will not play any role in what follows,
except that the notation $\gamma_w$ for the curve defined in \eqref{gamxiw} and $T_{\gamma_w}$ for the associated $1$-current,
 cf. \eqref{Tlambda.def}, will be used below.

To start, for $w\in W$ we define
\begin{equation}
 \xi(w)=w_1\xi_1+\ldots  + w_\ell \xi_\ell\label{xiw}  
\end{equation}
where $\xi_j$ denote eigenfunctions of the Jacobi operator, see in particular \eqref{index.ell}.
We then define the curve $\gamma_w$ via
\begin{equation}
    \gamma_w(t):=\exp_{\gamma(t)} \xi(w)(t),\label{gamxiw}
\end{equation}
cf. \eqref{gamexp}.
We will always assume that $R$ is small enough that
\[
\| \xi(w)\|_{L^\infty} \le r_0, \qquad\mbox{ in other words, }\   \gamma_w \in H_-(r_0), \quad\mbox{ for }w\in W.
\]
With this in hand, we define $Q_{VW}:B_R^\ell\to \calR_1'$ via 
\[
Q_{VW}(w) := T_{\gamma_{w}}\quad\mbox{for}\;
w\in W.
\]
Then \eqref{0map} is immediate, and \eqref{saddlepoint2} follows directly from \eqref{gxi.length}.

The construction  of $P_{WV}$ is carried out in \cite{Mesaric}, Lemma 4.3, by designing an $\R^\ell$-valued $1$-form $\Phi$ such that 
\beq\label{Psmooth}
T_{\gamma_\xi}(\Phi) = ( (\xi, \xi_1)_{L^2}, \ldots,  (\xi, \xi_\ell)_{L^2})
\eeq
for $\xi\in L^2(N\Gamma)\cap Lip$, as long as $\| \xi \|_{L^\infty} \le r_0$, a condition that can be guaranteed by a suitable choice of $R$.
Here we recall that $\gamma_\xi$ is the curve given by \eqref{gamexp} and $T_{\gamma_\xi}$ is its associated 1-current.  We then simply define $P_{WV}(T)= T(\Phi)$.
With this choice, \eqref{inverse} follows directly from \eqref{xiw}, \eqref{gamxiw}, and \eqref{Psmooth}.

The hard part of the proof of Lemma
\ref{thm:mesaric} is the verification of \eqref{saddlepoint1}. This is carried out 
in Proposition 4.1 of \cite{Mesaric}, to which we refer for the details. We remark that
the main ideas in this proof, including the construction of $\Phi$,
are similar to elements in the proof of Theorem 5 in \cite{White}.

Note that we may shrink at will the parameter $R$ in the definition \eqref{W.def} of $W$,
and the conclusions of the lemma remain valid. 

\subsection{An $\ell$-parameter family of solutions of \eqref{glh}}

To prove Proposition \ref{prop:goodtrajectory}, we will define an $\ell$-parameter
family of solutions of \eqref{glh} for every sufficiently small $\ep>0$.
In the final step, given $\ep$ and $\tau$ (where 
ultimately we will take $\ep <\ep_0(\delta)$), we will choose from this family
one solution $u_\ep$
such that $E_\ep(u_\ep(t))\ge L-o(1)$ for all times $t\in [0,\tau]$.

The initial data for this family of solutions is provided by the following result.

\begin{lemma}\label{lem:data} There exist $R, \ep_1>0$ such that
for every $\ep\in (0,\ep_1)$ and $w\in B^\ell_{R}$,
there exists a function $U^{\ep,0}_w\in H^2(M)$ satisfying the conditions:
\begin{enumerate}
\item
$w\mapsto U^{\ep,0}_w \mbox{ is Lipschitz continuous from }B^\ell_{R}\mbox{ into }H^2(M)$
\item
$\| U^{\ep,0}_w\|_{L^\infty} \le 1$ and $\| U^{\ep,0}_w\|_{H^2}\le C_\ep \qquad\mbox{ for all }\ep\in (0,\ep_1), w\in B^\ell_R,$\vspace{2pt}
\item
$\EUep( U^{\ep,0}_w) \le   L- c_0 |w|^2 + o(1)\quad\mbox{ as }\ep \to 0$\vspace{2pt}
\item
$\| \frac 1 \pi \star JU^{\ep,0}_w - T_{\gamma_{w}}\|_V \to 0$\vspace{2pt}
\end{enumerate}
uniformly for $w\in B^{\ell}_{R}$.
\end{lemma}

For fixed $w$, in view of  \eqref{gxi.length} and the construction of $\gamma_w$, conclusions
(3) and (4) hold if $U^{\ep,0}_w$ is a recovery sequence for the current $T_{\gamma_w}$
the Gamma-limit in Theorem \ref{thm:Mesaric-Gamma}. Such constructions are rather standard.
It is easy to arrange that $\| U^{\ep,0}_w\|_{L^\infty}\le 1$. The only points requiring attention are that
the construction has to be carried out so that it depends continuously on $w$, in the $H^1$ norm, and with some control over the $H^2$ norm. The former point is carried out in \cite{Mesaric}, and the latter can be achieved by a small modification of the construction of \cite{Mesaric}. We defer a more detailed discussion to Appendix \ref{App.A}.

The $H^2$ estimate facilitates the proof of Lemma \ref{lem:contflow} below, whose need arises because we require the Ginzburg-Landau heat flow rather than the pseudo-gradient flows employed in \cite{JSt, Mesaric}.

 \vskip.1in
 
Having constructed appropriate initial data for the Ginzburg-Landau flow, we are now ready to define the flow that we will use in our arguments below.

\begin{lemma}\label{lem:contflow}
For $\ep\in (0,\ep_1)$ and $w\in W$,
let $U^{\ep,1}_w(x,t)$ 
solve the Ginzburg-Landau heat flow with initial data $U^{\ep,0}_w$.
Then
\[
(t,w)\in [0,\infty)\times W \mapsto  U^{\ep,1}_w(\cdot, t) \in H^1(M;\C) \qquad\mbox{ is continuous}.
\]
\end{lemma}

See Appendix \ref{App.A} for the proof, which involves rather standard parabolic estimates.
\vskip.1in

Finally, we define $U_\ep: [0,\infty)\times W\to H^1(M;\C)$ by 
\begin{equation}
 U_\ep(t,w) := U^{\ep,1}_w(\cdot , \chi(w) t)  . \label{theflow}
\end{equation}
for a smooth, compactly supported $\chi:B^\ell_R\to [0,1]$ such that $\chi = 1$ in $B^\ell_{R/2}$.
We point out that
\beq\label{fix.boundary}
\mbox{ if }|w|=R, \qquad \mbox{ then }\quad U_\ep(t,w)  = U^{\ep,0}_w \quad\mbox{ for all }t\ge 0.
\eeq

We can guarantee that   $E^\ep_U(U_\ep(0,w)) < c_\ep - \delta$ whenever $|w|> R/2$ and
$\ep$ is small enough, for suitable $\delta$.

\subsection{Choosing a good trajectory}

We finally make use of the 
asymptotic saddle point geometry of $\EUep$, inherited from $E_V$
via the Gamma-convergence Theorem \ref{thm:Mesaric-Gamma}, to
complete the proof of Proposition \ref{prop:goodtrajectory}.

We will use the notation
\[
P_{WU}(u) = P_{WV}\big(\frac 1\pi \star Ju\big) \qquad \mbox{for}\;u\in H^1(M;\C).
\]

\begin{lemma}
There exist $\delta_0>0$ and $R_0>0$ such that for every
$R\in (0,R_0)$, there is some $\delta = \delta(R)>0$ such that if we define 
\begin{align*}
a_\ep&:= \max\{ \EUep(U^{\ep,0}_w)  : |w| = R \} \\
c_\ep&:= \min\{ \EUep(u)  : P_{WU}(u)= 0, \| \frac 1 \pi \star Ju  - T_\gamma\|_V \le
\delta_0 \} \\
d_\ep&:= \max\{ \EUep(U^{\ep,0}_w)  : |w| \le R \}  ,
\end{align*}
then
\beq\label{Limits}
a_\ep \to L-2\delta, \qquad 
\liminf_{\ep\to 0} c_\ep \ge L, 
\qquad d_\ep \to L
\eeq 
as $\ep\to 0$, where $L$ is the length of the geodesic $\gamma$.
\label{lem:acd}\end{lemma}

The assertions about $a_\ep$ and $d_\ep$ follow directly from \eqref{gxi.length} and 
Lemma \ref{lem:data},
and 
Step 3 of the proof of Theorem 4.4 of \cite{JSt} shows exactly that
$c_\ep\to L$. 
The proof uses only ingredients that we have collected in Theorem \ref{thm:Mesaric-Gamma} and Lemma \ref{thm:mesaric} below. 

Below we will not refer explicitly to the assertion about $\lim_{\ep\to 0}a_\ep$, but it plays a role in the proof of Lemma \ref{lem:goodtrajectory}, and together with the lower bound for $\liminf c_\ep$, it  reflects the asymptotic minmax geometry of $\EUep$.

Proposition  \ref{prop:goodtrajectory} will essentially follow from the next fact.

\begin{lemma}\label{lem:goodtrajectory}For each $r>0$ there exists $\ep_0>0$ 
and $R>0$ such that for every $0<\ep<\ep_0$ and
every $\tau>0$, there exists $w = w(\ep,\tau)$ such that
\beq\label{goodtraj1}
P_{WU}(U_\ep(\tau,w))= 0, \qquad \| \frac 1 \pi \star JU_\ep(\tau,w) - J_\gamma\|_V\le r.
\eeq
As a result, $w = w(\ep,\tau)$ satisfies 
\beq\label{Sandwich}
d_\ep \ge \EUep(U_\ep(t,w))\ge \EUep(U_\ep(\tau,w))\ge c_\ep\quad\mbox { for all } t\in [0,\tau] \mbox{ and }\ep\in (0,\ep_{0}).
\eeq
Finally, $w(\ep,\tau)\to 0$ as $\ep\to 0$.
\end{lemma}

\begin{proof}

One may prove \eqref{goodtraj1} by simply repeating the arguments
from Steps 5-8 of the proof of Theorem $4.4$ in \cite{JSt}, for $r >0$ such that $0< r<\delta_{0}$, where
$\delta_{0}$ is the constant from item $\eqref{saddlepoint1}$ of Lemma \ref{thm:mesaric}. Some comments are in order:

First, the argument in \cite{JSt} is stated for a pseudo-gradient flow (see Lemma 4.8, \cite{JSt}) with certain properties that our flow $(t,w)\to U_\ep(t,w)$ does not
possess. These are in fact not needed for the proof of \eqref{goodtraj1}, and some of them may appear in \cite{JSt} only because there Lemma 4.8 is quoted directly from a standard
text, which provides more than is actually needed.
These are the only properties of the flow that are  required for the proof of \eqref{goodtraj1}:
\begin{itemize}
    \item $t\mapsto E_\ep(U_\ep(t,w))$ is nonincreasing.
    \item $t\mapsto U_\ep(t,w)$ is constant for $w\in \partial W$, see \eqref{fix.boundary}.
    \item continuity properties of the flow, as summarized in Lemma \ref{lem:contflow}.
\end{itemize}
All of these are available here.

Without going into detail, we remark that the basic strategy of the proof is to apply degree theory arguments
to the maps $w\mapsto P_{WU}(U_\ep(t,w)):W\to W$ as $t$ varies from $0$ to $\tau$ (where $\tau=1$ in
\cite{JSt}, a harmless normalization).

Next, \eqref{Sandwich} follows directly from \eqref{goodtraj1} and Lemma \ref{lem:acd}.
Finally, we deduce from \eqref{Sandwich} and Lemma \ref{lem:acd} that
\[
 E_{\ep}\left(U_{\ep}(0,w)\right) = E_{\ep}\left(U^{\ep,0}_w\right) \to L
\qquad\mbox{ as $\ep\to 0$}.
\] 
Then conclusion (3) of Lemma \ref{lem:data}
implies that $w = w(\ep,\tau)\to 0$ as $\ep\to 0$.
\end{proof}

We are now in a position to present:

\begin{proof}[Proof of Proposition  \ref{prop:goodtrajectory}]
Let $0<\delta<\delta_{0}$ be given where $\delta_{0}$ is as in item $\eqref{saddlepoint1}$ of Lemma \ref{thm:mesaric}.
We take $\ep_{0}(\delta)$ and $R(\delta)$ as defined in Lemma $\ref{lem:goodtrajectory}$ and set
$u_{\ep}(x,t;\delta,\tau)=U_{\ep}(t,w(\ep,\tau))$.
By shrinking $\ep_0$ if necessary, we may ensure that $\left|w(\ep,\tau)\right|<\frac{R(\delta)}{2}$, as shown in
Lemma $\ref{lem:goodtrajectory}$.
Thus, $u_{\ep}(x,t;\delta,\tau)$ solves $\eqref{glh}$ since
\[U_{\ep}(t,w(\ep,\tau))=U_{w}^{\ep,1}(t,x)\]
for such $\ep>0$.

Since $\| U^{\ep,0}_w\|_{L\infty}\le 1$, it is clear that $|u_\ep(x,t)|\le 1$ everywhere,
and all other conclusions of the Proposition follow directly from Lemmas \ref{lem:acd}
and \ref{lem:goodtrajectory}.

\end{proof}

\section{Proof of  the main result}\label{sec:5}

The main result of this paper, stated more informally in the introduction  as Theorem \ref{thm:1a}, can now be phrased precisely as
\begin{theorem}\label{thm:1}
Let $(M,g)$ be a closed oriented $3$-dimensional Riemannian manifold, and 
let $\gamma$ be a closed, embedded, nondegenerate geodesic of length $L$.
Assume in addition that 
$\gamma = \partial S$ (in the sense of 
Stokes' Theorem) for some $2$-dimensional submanifold $S$ of $M$. 

Then for every $r>0$, there exists $\ep_1(r)>0$ such that if $0<\ep <\ep_1(r)$,
then there is a solution $u_\ep$ of the Ginzburg-Landau equations
\beq\label{PGL}
 - \Delta u_\ep + \frac 1{\ep^2}(|u_\ep|^2-1)u_\ep = 0 \qquad\mbox{ on } M
\eeq
such that
\[
\| \frac 1 \pi \star Ju_\ep - T_\gamma \|_{\calF}  \le r
\]
and
\[
\left| \EUep(u_\ep) - L \right| < r.
\]
As a result, there exists a sequence $(u_\ep)_{\ep>0}\subset H^1(M;\C)$
of solutions of the Ginzburg-Landau equations
such that 
\beq\label{t1a.conclusion}
\| \frac 1 \pi\star J u_\ep - T_\gamma\|_\calF \to 0, \qquad 
\EUep(u_\ep) \to L\qquad\qquad\mbox{ as $\ep\to 0$.}
\eeq
\end{theorem}

We remark that standard Gamma-convergence results (see Theorem \ref{thm:Mesaric-Gamma}) imply that
the sequence of solutions in \eqref{t1a.conclusion} satisfies
\beq\label{abitbetter}
\frac {e_\ep(u_\ep)}{\pi\logeps} \to \| T_\gamma\| = \calH^1\rest \Gamma \quad\mbox{ weakly as measures},
\eeq
which is the last conclusion of Theorem \ref{thm:1a}. Indeed, since $\EUep(u_\ep)$ is uniformly bounded,
there exists some measure $\mu$ such that 
\[
\frac {e_\ep(u_\ep)}{\pi\logeps} \to \mu\quad\mbox{ weakly as measures},
\]
after perhaps passing to a subsequence, and $\mu(M) = \lim_{\ep\to 0} \EUep(u_\ep) = L$. Standard Gamma-convergence results and \eqref{t1a.conclusion} imply
\[
\mu \ge \| T_\gamma\|,
\]
and since $\| T_\gamma\|(M) = L = \mu(M)$, it follows that $\mu= \|T_\gamma\|$, proving \eqref{abitbetter}.

\begin{proof}
The proof relies on an improvement on the properties of the flow defined in the previous section. The assertion is that the trajectory solving the Ginzburg-Landau flow identified in Proposition \ref{prop:goodtrajectory} remains close to $T_{\gamma}$ in the flat norm. More precisely, we will show:
\vskip.1in
\noindent{\bf Claim:}\\
For every $r>0$, there exist positive constants $\delta_1(r)$ and $\ep_1(r) < \ep_0(\delta_1) $, depending on $r$, 
such that  for every $\ep \in (0,\ep_1)$ and $\delta\in (0,\delta_1)$, and for every $\tau>0$, if $u_\ep = u_\ep(x,t;\delta,\tau)$ is the solution of
\eqref{glh} found in Proposition \ref{prop:goodtrajectory}, then
\begin{equation}
\|\frac 1\pi \star Ju_\ep(t) - T_\gamma\|_\calF < r \qquad\mbox{ for all }t\in [0,\tau].    \label{alwaysclose}
\end{equation}

To establish this, we assume toward a contradiction that there exists $r>0$ and sequences $ \delta_k, \ep_k\to 0$
and $0<t_k \le \tau_k$ such that  $u_k(x,t) := u_{\ep_k}(x,t; \delta_k,\tau_k)$ satisfies
\begin{equation}
\| \frac 1\pi \star Ju_k(\cdot, t_k) - T_\gamma\|_\calF = r .    \label{contra1}
\end{equation}
Clearly, we may assume that $r<r_0$ for some $r_0$ to be chosen below.

\medskip
\noindent {\it Step 1:}
Under the assumption \eqref{contra1}, we will argue that necessarily $t_k\to \infty$ as $k\to \infty$.
(In fact, below we only need to know that $t_k$ is bounded away from $0$.)

Assume toward a contradiction that  
$\liminf_k\   t_k <  K$, for some $K>0$.
By passing to subsequences, relabelling, and invoking standard compactness,
continuity, and Gamma-convergence results ({\em i.e.} Theorem \ref{thm:Mesaric-Gamma}), we may assume that
the following hold. 

First, $t_k\le K$ for all $k$.

Second, there exists a $1$-current
$J_1\in \calR_1\cap \calF_1'(M)$ such that
\[
\|\frac 1 \pi \star Ju_k(\cdot, t_k)  - J_1\|_\calF \to 0.
\]
Hence, by \eqref{contra1},
\beq\label{J0-J1}
\| J_1- T_\gamma\|_\calF = r.
\eeq
Third, there exists a Radon measure $\mu_1$ such that
\[
\frac {e_{\ep_k}(u_k(\cdot, t_k)) }  { \pi\logeps}\rightharpoonup \mu_1,
\qquad\mbox{weakly as measures.}
\]
and in addition
\beq\label{mugeJ}
\mu_1\ge \|J_1\|,
\eeq
cf. \eqref{tovar}.

Finally, since $\delta_k\to 0$ and $M$ is compact, it follows from 
Proposition \ref{prop:goodtrajectory} that
\beq\label{massL}
\mu_1(M) 
 = L = \| T_\gamma\|(M).
\eeq

We next claim that
\[
\frac {e_{\ep_k}(u_k^0(\cdot)) }  { \pi\logeps}\rightharpoonup \mu_0 = \|T_\gamma\|,
\qquad\mbox{weakly as measures.}
\]
Indeed, we may assume that $\frac {e_{\ep_k}(u_k^0(\cdot)) }  { \pi\logeps}$ converges weakly as measures to a limit $\mu_0$ as $k\to \infty$.  Then
recalling that $\| \frac 1\pi \star Ju_{k}^0 - T_\gamma\|_\calF \le \delta_k\to 0$,
standard Gamma-convergence results as in \eqref{mugeJ} imply that
$\mu_0\ge \| T_\gamma\|$. On the other hand, as in \eqref{massL}, 
\[
\mu_0(M)  = L = \mbox{length}(\gamma) = \|T_\gamma\|(M),
\]
so it follows that in fact $\mu_0 = \|T_\gamma\|$ as claimed.
It then follows from \eqref{J0-J1} and \eqref{mugeJ} that $\mu_1\ne \mu_0 = \|T_\gamma\|$.

We will obtain a contradiction, completing Step 1, by showing that 
under our assumptions $\mu_0$ and $\mu_1$ must be equal. 
Indeed, after taking the inner product of \eqref{glh} with $\pp_t u_k$, 
standard computations show that
\[
\pp_t e_{\ep_k}(u_k) = - |\pp_t u_k|^2 + \mbox{div}(\pp_t u_k \cdot \nabla u_k).
\]
Multiplying by a function $\phi\in C^1(M)$ and integrating by parts,
\begin{equation}\label{tk-sk1}
-
\left. \int_M \phi e_{\ep_k}(u_k)\vol\right|_{0}^{t_k}
=\int_{0}^{t_k}\int _M \phi |\pp_t u_k|^2  - (\nabla \phi , \pp_t u_k\cdot\nabla u_k)_g\, \vol\,dt
\end{equation}
Clearly
\[
\frac 1{\pi|\log\ep_k|}\left. \int_M \phi e_{\ep_k}(u_k)\vol\right|_{0}^{t_k}
\rightarrow\int_M \phi \ (\mu_0  - \mu_1).
\]
On the other hand, it is not hard to see that after dividing by $\pi |\log\ep_k|$, 
the right-hand side of \eqref{tk-sk1} tends to $0$ as $k\to\infty$.
First, taking $\phi=1$ in \eqref{tk-sk1}, we see that 
\[
 \frac 1{\pi |\log\ep_k|}\int_{0}^{t_k}\int_M   \,|\pp_t u_k|^2 \vol \, dt
 = E_{\ep_k}(u_k^0) - E_{\ep_k}(u_k(\cdot, t_k))   \le 2\delta_k\to 0
\]
as $k\to \infty$. It immediately follows that
\begin{align*}
\left| \int_{0}^{t_k}\int_M  \phi \frac{ |\pp_t u_k|^2 }{\pi |\log\ep|}\vol\,dt \right|
\to 0 \qquad \mbox{ as }k\to \infty
\end{align*}
Similarly, from Cauchy-Schwarz, the fact that
\[
\int_M \frac {|\nabla u_k(\cdot, t)|^2}{\pi|\log\ep_k|} \vol\, dt \le E_{\ep_k}(u_k(\cdot, t)) \le L+ \delta_k\qquad\mbox{ for all $t\ge 0$}
\]
and the assumption that  $t_k\le K$,
we easily see that
\[
\frac 1{\pi|\log \ep_k|}\left|
\int_{0}^{t_k}\int_M   (\nabla \phi , \pp_t u_k\cdot\nabla u_k)_g\, \vol\,dt \right|
\to 0 \qquad\mbox{ as }k\to\infty.
\]
Combining these, we conclude that $\int_M \phi (\mu_1-\mu_0) = 0$
for all $\phi\in C^1(M)$, and hence that $\mu_0=\mu_1$.
This contradiction yields the conclusion that if \eqref{contra1}  holds, then $t_k\to\infty$, completing Step 1.
\vskip.1in
\noindent{\it Step 2}: We will now reach a contradiction to \eqref{contra1}, and so obtain Claim \ref{alwaysclose}.

To this end, we apply Theorem \ref{T:BOS} to the sequence of functions
\[
\tilde u_k(x,t) = u_k(x, t+t_k-1),
\]
which, in particular, satisfies \eqref{glh} on $M\times [0,\infty)$ with $\ep = \ep_k\to 0$.
This yields a function $\Phi_*:M\times (0,1]\to \R$ that solves the heat
equation, as well as measures $(\mu^t_*)_{0<t\leq 1}$
and $(\nu^t_*)_{0<t\leq 1}$ such that
\[
\frac {e_{\ep_k}(\tilde u_k(\cdot, t)) }
{\pi|\log\ep_k|}\vol\rightharpoonup
\mu^t_* = \frac 12 |\nabla \Phi_*(\cdot ,t)|^2\vol + \nu^t_*
\]
weakly as measures and $(\nu^t_*)_{0<t\leq 1}$ is a $1$- dimensional Brakke flow
satisfying \eqref{nustar.form} and \eqref{ldb} (with $n=2$).

As with \eqref{massL}, and using Claim 1, it follows that
\beq\label{mut.const}
\mu^t_*(M)=L\qquad\mbox{ for all }t\in (0,1].
\eeq
We recall the standard estimate
\beq\label{spe}
\frac d{dt}\int_M \frac{|\nabla\Phi_*(\cdot, t)|^2}2\vol
=  - \int_M |\pp_t\Phi_*|^2\, dx\le 0
\eeq
(the counterpart for the linear heat equation of \eqref{tk-sk1}).
Since $t\mapsto \nu^t_*(M)$ is also nonincreasing, as noted in \eqref{nut.decr},
we conclude from \eqref{mut.const} that both 
\[
 \nu^t_*(M)\quad \mbox{ and }\quad\int_M \frac{|\nabla\Phi_*(\cdot, t)|^2}2\vol  \qquad\mbox{ are both
independent of $t\in (0,1]$}.
\]
It then follows from \eqref{spe} that $\pp_t\Phi_* = 0$ and hence that $\Phi_* = \Phi_*(x)$
is independent of $t$ and harmonic. Similarly, it follows from 
\eqref{nut.const}, \eqref{nut.stationary} that there exists a stationary
$1$-dimensional varifold $V_*$ such that
\[
\nu^t_*=V_*\qquad\mbox{ for all }\;t\in (0,1],
\]
and then by continuity at $t=1$ as well. Also,  \eqref{nustar.form} and \eqref{ldb} imply that there exists a $1$-rectifiable
set $\Sigma_{*}\subset M$ and a function $\Theta_*$ such that
\beq\label{formofVstar}
V_* = \Theta_*(x)\calH^1\rest \Sigma_{*}, \qquad\qquad \Theta_*\ge \eta>0   \quad \calH^1 \ 
\mbox{{\em a.e.}  in }\Sigma_{*}.
\eeq
Moreover, as in the proof of Claim 1, there exists a $1$-current $J_1\in \calR_1\cap \calF_1'(M)$ such that
\[
\| \frac 1 \pi \star J\tilde u(\cdot, 1) - J_1\|_\calF = \| \frac 1 \pi \star J u(\cdot, t_k) - J_1\|_\calF \to 0
\]
and thus
\beq\label{J0.p1}
\| J_1 - T_\gamma\|_\calF=r.
\eeq
We claim that in addition
\beq\label{J0.p2}
V_* \ge \| J_1\|, \qquad V_*(M)\le L.
\eeq
The second assertion follows from \eqref{mut.const}, and
the first assertion is a consequence of standard Gamma-convergence results,
which imply that
\[
\mu^1_* = \frac 12| \nabla \Phi_*(x)|^2\vol + V_* \ge \| J_1\|.
\]
Since $\frac 12| \nabla \Phi_*(x)|^2\vol $ is absolutely continuous with respect to
$\vol$ and $\|J_1\|$ is concentrated on a $1$-rectifiable set, this implies
that $V_* \ge \| J_1\|$, as claimed.

However, through an appeal to Proposition \ref{lem:noVstar}, we see that 
no such varifold can exist. Claim \ref{alwaysclose} is established.
\vskip.1in
\noindent{\it Step 3}:
Fix $r>0$, and let $\delta_1(r)$ and $\ep_1(r)$ be as provided in Claim \ref{alwaysclose}. We may assume that $\delta_1(r)\le r$. 
For $\ep \in (0,\ep_1(r))$, let
\[
u_k = u_{\ep}(\cdot, \cdot; \frac 12 \delta_1(r), 2^k).
\]
It follows from Proposition \ref{prop:goodtrajectory} that
\begin{align*}
\delta_1(r) 
&\ge
E^\ep_U(u_k(\cdot, 0)) - 
E^\ep_U(u_k(\cdot, 2^k)) 
 =
\frac 1 {\pi|\log\ep|} \int_0^{2^k} \int_M |\pp_t u_k|^2 \vol\, dt,
\end{align*}
so there exists $\sigma_k\in (0,2^k)$ such that $w_k := u_k(\cdot, \sigma_k)$ satisfies
\[
\int_M |\Delta w_k - \frac 1{\ep^2}(|w_k|^2-1)w_k |^2\, \vol
=
\left. \int_M |\pp_t u_k|^2\, \vol\right|_{t=\sigma_k} 
\le \delta_1  \pi |\log\ep|2^{-k}.
\]
Also, it follows from Claim \ref{alwaysclose} that
\[
\| \frac 1\pi \star Jw_k - T_\gamma\|_\calF = \| \frac 1\pi \star Ju_k(\cdot, \sigma_k) - T_\gamma\|_\calF< r.
\]
Since $|u_k|\le 1$ everywhere, we have that
$\| \Delta w_k\|_{L^2(M)} \le C_\ep$, and hence by elliptic regularity, 
$\| w_k \|_{H^2} \le C_\ep$. 
One may thus extract a subsequence and a function $u_\ep\in H^2(M;\C)$ such that
$w_k\to u_\ep$ weakly in $H^2$, and it easily follows from the above that
\[
-\Delta u_\ep + \frac 1{\ep^2}(|u_\ep|^2-1)u_\ep = 0 \qquad\mbox{ and }\quad 
 \| \frac 1\pi \star Ju_\ep - T_\gamma\|_\calF\le r.
\]
Finally, we may insist that $\delta<r$, and then it follows from Proposition 
\ref{prop:goodtrajectory} that $|\EUep(u_\ep)- L|< r$.
\end{proof}

\appendix
\section{Appendix}\label{App.A}

\subsection{On the proof of Lemma \ref{lem:data}}

As remarked above, this lemma is essentially proved in \cite{Mesaric}. We describe the 
proof given there and the extremely small modifications that we need.

The idea of the proof is first to construct $U^{\ep,0}_0$, with its vorticity concentrating around the central geodesic $\Gamma$, then 
for  $w\in W$, to define
\beq\label{form.of.Uepw}
U^{\ep,0}_w := U^{\ep,0}_0\circ{}O_{w}^{-1}
\eeq
where $O_w:M\to M$ is a suitable family of diffeomorphisms indexed by $w\in W$ such that 
$(w,x)\to O_w(x)$ is smooth, described below.
\vskip.1in
{\bf Construction of $U^{\ep,0}_0$}
Recall that in \eqref{y.tau.def} we defined a map $y: K_{r_0}\to \R^2$, smooth and nonvanishing away from  $\Gamma$. Let $y^{0}:M\setminus K_{r_0/2}\to S^1$ be any smooth function
such that $y^{0}(x)=y(x)\slash\left|y(x)\right|$ in 
$K_{r_{0}}\setminus K_{r_{0}/2} $.
The existence of such a function is a consequence of the topological
assumption \eqref{gambdS}. 

Then we set $\tilde{v}^{\ep}:\mathbb{R}^{2}\to\mathbb{R}^{2}$ by
$\tilde{v}^{\ep}(p)=f\left(\frac{\left|p\right|}{\ep}\right)\frac{p}{\left|p\right|}$ where
$f:[0,\infty)\to[0,1]$ is a smooth nondecreasing function such that $f(s)=s$ for $s\in[0,1\slash2]$ and $f(s)=1$ for $s\ge1$.
Finally we define
\begin{equation*}
U^{\ep,0}_0(x) :=
\begin{cases}
\tilde{v}^{\ep}(y(x))& \text{for }x\in{}K_{r_{0}},\\
y^{0}(x)& \text{for }x\in{}M\setminus{}K_{r_{0}}.
\end{cases}
\end{equation*}
{\bf The only way} in which this construction differs from that in \cite{Mesaric}
is that there, $f$ is chosen to be $f(s) = \min(s,1)$, which is Lipschitz continuous
but not smooth. With this change,  $U^{\ep,0}_0$ is smooth.

\vskip.1in
{\bf Construction of $O_w$}. We take $O_w$ in \eqref{form.of.Uepw} to be {\em exactly} the
same map as in \cite{Mesaric}, see  pg. 62.
\vskip.1in
The construction easily implies that $(w,x)\mapsto U^{\ep,0}_w(x)$ is smooth
and hence that $\|U^{\ep,0}_w\|_{H^2}\le C_\ep $ for all $w\in W$. All other conclusions
are proved in \cite{Mesaric}, and some are obvious anyway, such as that 
$\| U^{\ep,0}_w\|_{L^\infty}\le 1.$
In particular,  $(3)$, which follows from a Gamma-limsup type estimate together with
\eqref{gxi.length}, is verified in  Lemma $5.5$ from \cite{Mesaric}.
Finally, $(4)$ follows from Lemma $5.4$ of \cite{Mesaric}.

\subsection{Proof of Lemma \ref{lem:contflow}}

\begin{proof}
The maximum principle and standard energy estimates imply that for every $t>0$,
\begin{align*}
\| U^{\ep,1}_w(\cdot, t)\|_{L^\infty(M)}& \le 1, \\
E_\ep(U^{\ep,1}_w(\cdot, t)) + \frac 1{\pi\logeps} \int_0^t\int_M |\pp_t U^{\ep,1}_w|^2\, d\vol \, dx &\le E_\ep(U^{\ep,1}_w(\cdot, 0)) \le L+1
\end{align*}
for all $|w|\le R$, provided $\ep$ and $R$ are small enough.
We next claim that for every $t>0$, there exists $C = C_{\ep, \tau}$ such that
\beq\label{H2est}
\| U^{\ep,1}_w(\cdot, t)\|_{H^2} \le C_{\ep, \tau} \qquad\mbox{ for all }t\in [0,\tau]\mbox{ and }|w|\le R.
\eeq
To specify the norm, we fix an open cover
$\{U_j\}_{ j\in J}$ of $M$, with local coordinates $\varphi_j: U_j\to V_j\subset\R^3$ on each patch, 
and a finite partition of unity $\{\eta_j\}$ subordinate to $\{ U_j\}$.
We then define
\[
\| u\|_{H^2}^2 = \| u\|_{L^2}^2 + \|  |\nabla u|_g\|_{L^2}^2 + \sum_{j\in J}\sum_{k=1}^3 \| \sqrt{\eta_j}  |\nabla\partial_k( u \circ \varphi_j^{-1})|_g\|^2_{L^2(V_j)},
\]
where $\partial_k$ denotes differentiation with respect to local coordinates on $V_j$.
To prove \eqref{H2est}, we write \eqref{glh} in local coordinates on each patch, apply $\pp_k$ to derive an equation for $V_k := \pp_kU^{\ep,1}_w$ of the form
\[
\pp_t V_k - \Delta_g V_k = \mbox{ terms involving $U^{\ep,1}_w, \nabla U^{\ep,1}_w$}.
\]
Multiplying by $\pp_t V_k$, using the fact from Lemma \ref{lem:data} point (2) that $\|\nabla V_k(\cdot, 0) \| \le C_\ep$, integrating by parts, and carrying out rather standard estimates leads to \eqref{H2est}.

It follows from the above estimates and the equation that 
$\| \pp_t U^{\ep,1}_w(\cdot, t)\|_{L^2} \le C_{\ep, \tau}$ for $0< s \le \tau$.
Thus, for $0\le t_{1}< t_{2}\le \tau$ and any $w\in{}W$, we have
\begin{align*}
\| U^{\ep,1}_w(\cdot, t_2) -  U^{\ep,1}_w(\cdot, t_2) \|_{L^2}^2\
&\le(t_{2}-t_{1})\int_{M\times[t_{1},t_{2}]}\!{}\left|\partial_{t}U_{w}^{\varepsilon,1}(x,t)\right|^{2}\vol\ {}dt
\\
&
\le C_{\ep, \tau}(t_{2}-t_{1}).
\end{align*}
Then the interpolation estimate $\|u\|_{H^1} \le C \|u\|_{L^2}^{1/2}\|u\|_{H^2}^{1/2}$
and \eqref{H2est}  imply that
\[
\left\|U_{w}^{\varepsilon,1}(\cdot,t_2)-U_{w}^{\varepsilon,1}(\cdot,t_1)\right\|_{H^1 } \le C_{\ep,\tau} \sqrt{t_2-t_1} \quad\mbox{ for }w\in W, \ 0\le t_1<t_2\le \tau.
\]
Now consider $w_{1},w_{2}\in{}W$. Writing $f_\ep(u) = \frac 1{\ep^2}(1-|u|^2)u$ and using the identity
\[
f_\ep( b) - f_\ep ( a) 
= \int_0^1\frac d{d\sigma} f_\ep(\sigma  b + (1-\sigma) a)d\sigma =\int_0^1f_\ep'(\sigma  b + (1-\sigma) a)d\sigma\  ( b- a)\ ,
\]
we find that  $V:= U_{w_{2}}^{\varepsilon,1}-U_{w_{1}}^{\varepsilon,1}$ satisfies the equation
\[
\pp_tV - \Delta V = g V,
\]
where  $\|g(\cdot, t)\|_{L^\infty} \le C$ for every $t$.  In addition, it follows from Lemma \ref{lem:data} that 
$\| V(\cdot, 0)\|_{H^1} \le C|w_2-w_1|$. Thus carrying out further standard parabolic estimates
(multiplying by $V$ or $\pp_t V$, integrating by parts ...) leads to 
\[
\| U^{\ep,1}_{w_2} (\cdot, t)- U^{\ep,1}_{w_1}(\cdot, t) \|_{H^1} \le C_{\ep,\tau} |w_2-w_1| \qquad
\mbox{ for }0\le t\le \tau.
\]
We conclude that  the map $(t,w)\in[0,\tau]\times{}W\mapsto{}U_{w}^{\varepsilon,1}(\cdot,t)\in{}H^{1}(M;\mathbb{C})$ is continuous, since it is separately uniformly 
continuous in $t$ and $w$.
A more detailed reference for such parabolic estimates on manifolds can be found, e.g. in  appendix A of \cite{mant}.
\end{proof}


\medskip\noindent
{\it Acknowledgements.}  The work of A.C. and R.J. on this project was partly supported 
by the Natural Sciences and Engineering Research Council of Canada under Operating Grant 261955. P.S. gratefully acknowledges that this research was supported by the Fields Institute for Research in Mathematical Sciences and by a Simons Collaboration grant 585520.
The contents of this paper are solely the responsibility of the authors and do not necessarily represent the official views of any of the organizations mentioned above.

\bibliographystyle{acm}
\bibliography{CJS}

\end{document}